\documentclass[10pt]{article}

\usepackage{amsthm,amsfonts,amsmath,amscd,amssymb,times, epsfig,verbatim}
\usepackage[all]{xy}
\usepackage{enumerate,array}

\newtheorem{thm}{Theorem}
\newtheorem{prop}{Proposition}
\newtheorem{lem}{Lemma}
\newtheorem{cor}{Corollary}

\theoremstyle{remark}
\newtheorem{rem}{Remark}

\newcommand{\C}{\mathbb{C}}

\newcommand{\OO}{\mathcal{ O}}

\newcommand{\R}{\mathbb{ R}}

\newcommand{\T}{\operatorname{T}}

\newcommand{\Z}{\mathbb{ Z}}

\DeclareMathOperator{\dummyO}{O}
\renewcommand{\O}{\dummyO}
\DeclareMathOperator{\dummyo}{o}
\renewcommand{\o}{\dummyo}

\DeclareMathOperator{\dummygg}{\mathfrak{g}}
\renewcommand{\gg}{\dummygg}

\DeclareMathOperator{\rank}{rank}
\DeclareMathOperator{\cone}{cone}
\DeclareMathOperator{\diag}{diag}

\title{Topology and geometry of the canonical action of $T^4$ on the complex Grassmannian $G_{4,2}$ and the complex projective space $\C P^{5}$}
\author{Victor M.~Buchstaber and Svjetlana Terzi\'c}

\begin{document}

\maketitle
\begin{abstract}

We consider the canonical action of the compact torus $T^4$ on the Grassmann manifold $G_{4,2}$ and prove that the orbit space $G_{4,2}/T^4$ is homeomorphic to the sphere $S^5$. We prove that the induced differentiable structure    
on $S^5$ is not the smooth one  and describe the smooth and the singular points. We also consider the action  of $T^4$ on
$\C P^5$ induced by the composition of the second symmetric power $T^4\subset T^6$  and the standard action of $T^6$ on $\C P^5$ and prove that the orbit space $\C P^5/T^4$ is homeomorphic to the join $\C P^2\ast S^2$. The Pl\"ucker embedding $G_{4,2}\subset \C P^5$ is equivariant for these actions and induces embedding $\C P^1\ast S^2 \subset \C P^2 \ast S^2$ for the standard embedding $\C P^1 \subset \C P^2$. 

All our constructions are compatible with the involution given by the complex conjugation and give the  corresponding results 
for real Grassmannian $G_{4,2}(\R )$ and real projective space $\R P^5$ for the action of the group $\Z _{2}^{4}$. We prove that the orbit space $G_{4,2}(\R )/\Z _{2}^{4}$ is homeomorphic to the sphere $S^4$ and that the orbit space $\R P^{5}/\Z _{2}^{4}$ is homeomorphic to the join $\R P^2\ast S^2$.
\footnote{keywords: torus action, orbit, space, Grassmann manifold, complex projective space}
\end{abstract}

\tableofcontents
\section{Introduction}

As the natural extension of the theories of toric and quasitoric manifolds, as well as the theory of symplectic manifolds with Hamiltonian compact torus action, we proposed  in~\cite{OB} the theory of $(2n,k)$-manifolds. This theory  studies the  smooth manifolds $M^{2n}$ with  a smooth action  of a compact torus $T^k$ and an analogous of a moment map $M^{2n}\to \R ^k$. The main examples of $(2n,k)$-manifolds are Grassmann manifolds and one of our key aims is to apply the achievements of toric geometry and topology for the solution of the well known problems on Grassmann manifolds.  In that context our aim in this paper is to extract the properties from the  crucial examples of such manifolds and to describe the corresponding orbit spaces. 

More precisely, our focus in this paper  is on the complex  Grassmann manifold $G_{4,2}$ and the complex projective space $\C P^5$ endowed with  the action of the compact torus $T^4$. In the case of $G_{4,2}$ we consider the canonical action of the torus $T^4$, while in the case of $\C P^5$ the action of $T^4$ is given by the composition of the second symmetric power $T^4\subset T^6$  and the standard action of $T^6$ on $\C P^5$. In both cases the action extends to the action of the algebraic torus $(\C ^{*})^4$. 

Our approach makes us possible to obtain the results about the geometry and topology of these actions using the methods of toric geometry and toric topology.  In connection with this we want to recall  the paper~\cite{GFMC}, where  it were introduced and studied  some polyhedra in the real Grassmann manifolds which are called Grassmannian simplices. For an arbitrary Grassmann manifold these polyhedra are obtained as the generalization of the standard simplices in the real projective space. We want here to emphasize that such simplices exist for any smooth toric manifold. Namely, for a smooth toric manifold $M^{2n}$ with the moment map $\mu : M^{2n}\to P^{n}$ there exists a continious section $s : P^{n}\to M^{2n}$, $\mu \circ s = I_{d}$, such that this section gives equivariant surjection 
\[
(T^{n}\times P^{n})\to M^{2n},\;\; (t,p) \to t\cdot s(p).
\]
In that sense the simplices considered in~\cite{GFMC} are the reflexion of the general fact in toric topology. 

Therefore, our aim is not only to obtain the results on geometry and topology of the compact torus action on a smooth manifold, but also to study them from the point of view of the existence of the corresponding sections. 

On the contrary to the case of toric manifold, where the section exists over the whole polytope $P^{n}$, it the case  of  real or complex Grassamnn manifold $G_{n,k}$ such section over $P^{n-1}$ which is hypersimplex $\Delta _{n,k}$  does not exist. Instead, in this paper we will consider the cell complex for $G_{4,2}$ consisting of the Grassmannian simplices in terminology of~\cite{GFMC} and  the question of the existence of the section over such complex. 

In the case of the Grassmannian $G_{4,2}$ we use the standard moment map arising from the Pl\"ucker embedding of $G_{4,2}$ in $\C P^5$. The image of this moment map in both cases is the octahedron $\Delta _{4,2}$.  The action of $(\C ^{*})^4$ allows us to use the classical results of algebraic geometry about  the image of the algebraic torus orbits by the moment map. 

The canonical action of the compact torus $T^n$ extends  to the action of the algebraic torus $(\C ^{*})^{n}$ on general Grassmann manifold $G_{n,k}$.  In connection with such action of algebraic torus it is defined the notion of the main stratum~\cite{GEL4}, it is the set $W\subset G_{n,k}$ consisting of the points  $X$  such that the image  by the standard moment map of the closure of the orbit $(\C ^{*})^{n}\cdot X$ is the whole hypersimplex $\Delta _{n,k}$. Nevertheless, the closed subset $G_{n,k} - W$ is  again invariant under the action of $(\C ^{*})^{n}$ and its image by the moment map is again the whole hypersimplex $\Delta _{n,k}$. The description of the action of $T^n$ on the main stratum $W$ is simple, since it is homeomorphic to $T^{n-1}\times \stackrel{\circ}{\Delta_{n,k}}\times F$, for some space $F$. Still the description of the action of $T^n$ on $G_{n,k}- W$ remains highly nontrivial problem. 
In this paper, using the methods of toric topology, we solve this problem in the case of Grassmannian $G_{4,2}$.

We prove that $G_{4,2}/T^3$ is homeomorphic to the join $\C P^1\ast S^2$, while $\C P^5/T^3$ is homeomorphic to the join 
$\C P^2\ast S^2$, where $S^2 = \partial \Delta _{4,2}$.  Moreover, we prove that $G_{4,2}/T^4$ is topological manifold without boundary, and, thus it is homeomorphic to $S^5$. In the same time we describe the smooth and singular points of the orbit space $G_{4,2}/T^4=S^5$. It gives that the projection $G_{4,2}\to S^5$ is not a smooth map for the unique, the standard one, differentiable structure on $S^5$.

The action of $T^4$ on $G_{4,2}$ and $\C P^{5}$ induces the action of the group $\Z _{2}^{4}$  on the real Grassmann manifold $G_{4,2}(\R )$ and real projective space $\R P^{5}$.  All our constructions in the complex case are compatible with the involution given by the complex conjugation and, thus,  give the  corresponding results in the real case.  We prove that the orbit space $G_{4,2}(\R )/\Z _{2}^{4}$ is homeomorphic to the sphere $S^4$ and that the orbits space $\R P^{5}/\Z _{2}^{4}$ is homeomorphic to the join $\R P^2\ast S^2$. We also describe the smooth and the singular points of the orbit space $G_{4,2}(\R )/Z_{2}^{4} = S^4$. It implies that there is no differentiable structure on the sphere $S^4$  such that the natural projection $\pi : G_{4,2}(\R )\to S^{4}$ is a smooth map.

We want to note that  there is a series of papers~\cite{KT1},~\cite{KT2},\cite{KT3} about the Hamiltonian  action of the torus $T^{n-1}$ on a symplectic manifold $M^{2n}$. In our paper we have  such an action of the torus $T^3$ on $G_{4,2}$. But, we  also have the action of the torus $T^3$ on $\C P^5$, which is an example of Hamiltonian action of the torus $T^{n-2}$ on a symplectic manifold $M^{2n}$.   In that sense, we consider the case which is  studied in the mentioned series of papers by the methods of symplectic geometry, but also the more general case. We hope if develop the theory of Hamiltonian action of the torus $T^k$ on a symplectic manifold $M^{2n}$, where $k< n-1$, the example of the action of $T^3$ on $\C P^5$ studied in detail in our paper will be useful. 

In the paper~\cite{GFMC} it is considered   interpretation of the complex and real Grassmann manifold, $G_{n,k}(\C)$ and $G_{n,k}(\R)$, as configuration manifolds of $n$ points in $k$-dimensional complex or real vector space. Such interpretation is in~\cite{GFMC} further used  for the description of the action of the group $(\C ^{*})^{n}$ on $G_{n,k}(\C)$ and the action of the group  $(\R ^{*})^{n}$ on $G_{n,k}(\R )$, where it  is constructed the models for $G_{n,k}(\C)/(\C ^{*})^{n}$ and $G_{n,k}(\R )/(\R^{*})^{n}$ by proving that  they  naturally diffeomorphic to the spaces $C_{n,k}(\C)$ and $C_{n,k}(\R)$  of equivalence classes of generic configurations of $n$ points in $\C P^{k-1}$ or $\R P^{k-1}$.   Still the questions on cohomology structure or homotopy type of these orbit spaces are not completely answered.  The model spaces $C_{n,k}(\R)$   found further important applications for example in~\cite{TO}. In connection with this, our results about the orbit spaces of the action of compact torus $T^{n}\subset (\C ^{*})^{n}$ on $G_{n,k}(\C)$ and the action of $(\Z _{2})^{n}\subset (\R ^{*})^{n}$ on $G_{n,k}(\R)$ have applications as the results about configuration spaces of $n$ points in $k$-dimensional spaces.

\section{Generalities}
Denote by $G_{n,k}$ the Grassmann manifold of all $k$-dimensional subspaces in $\C ^n$. We consider  the canonical actions of algebraic torus $(\C ^{*})^{n}$ and compact torus
$T^{n}\subset (\C ^{*})^{n}$ on $G_{n,k}$. Note that $(\C ^{*})^{n}=T^{n}\times \R ^{n}_{+}$. 

Recall that the canonical action of $(\C ^{*})^{n}$ on $G_{n,k}$ is induced from the canonical  of $(\C ^{*})^{n}$ on $\C ^{n}$  which is given by
\[
(z_1,\ldots ,z_n)(v_1,\ldots ,v_n) = (z_1v_1,\ldots ,z_nv_n)
\]
for $(z_1,\ldots ,z_n)\in (\C ^{*})^n$ and $(v_1,\ldots ,v_n)\in \C ^n$. 

We denote by $\OO _{\C}(X)$ the  orbit of an element $X\in G_{n,k}$ under the action of the algebraic torus $(\C ^{*})^n$  and by $\overline{\OO _{\C} (X)}$ its closure in $G_{n,k}$. It is known that $\overline{\OO _{\C}(X)}$ is a compact algebraic manifold which consists of finitely many $(\C ^{*})^n$ - orbits and $\overline{\OO _{\C}(X)}$ is a toric manifold  in which $\OO _{\C}( X)$ is unique everywhere dense open orbit. 

Note that the action of $(\C ^{*})^n$ on $G_{n,k}$ is not effective, namely each point in $G_{n,k}$ is fixed by diagonal subgroup $\C ^{*}=\{(z,\ldots ,z) | z\in \C ^*\}$. Because of that we will further equally consider the effective action of $(\C^{*})^{n-1}\subset (\C ^{*})^{n}$ given by the embedding $(z_1,\ldots ,z_{n-1})\to (z_1,\ldots z_{n-1},1)$. 

\section{Pl\"ucker coordinates and moment map}
Recall the notion of the Pl\"ucker coordinates on Grassmann manifolds. We consider first  the Stiefel manifold $V_{n,k}$ of the orthonormal frames in $\C ^n$. There is the canonical embedding of $V_{n,k}$ into $\C ^{N}$, where $N=nk$ which to each $k$-reaper assigns $(n\times k)$-matrix $A$. This embedding is equivariant related to the action of the group $U(n)\times U(k)$ on $(n\times k)$-matrices, where the group $U(n)$ acts from the left, while the group $U(k)$ acts from the right. The Pl\"ucker map $V_{n,k} \to \C ^{N_1}$, where $N_1 = \frac{n!}{k!(n-k)!}$ assigns to each $(n\times k)$-matrix $A$ the collection of $(k\times k)$-minors $P^{J}(A)$, $J\subset \{1,\ldots ,n\}$, $|J|=k$ and induces the map from $V_{n,k}$ to $\C ^{N_1}-\{0\}$. Related to the right action of $U(k)$ on $\C ^{N}$ and coordinate wise action of the torus $T ^{N_1}$ on $\C ^{N_1}$ the Pl\"ucker map is equivariant. Passing to the quotient space we obtain an embedding of the Grassmann manifold $G_{n,k}$ into complex projective space $\C P^{N_1-1}$. This embedding is equivariant related to the representation $T^n\to T^{N_1}$, which is given by the $k$-th exterior power of the standard representation of $T^n$ into $\C ^n$.  

In order to make the exposition simpler we will use an arbitrary and not just an orthonormal frames for Grassmann manifolds.  

Using Pl\"ucker map it is defined   the moment map $\mu : G_{n,k} \to \R ^n$, first in~\cite{M} and~\cite{GFMC}, by
\begin{equation}\label{moment-map}
\mu (X) = \frac{\sum\limits_{J}|P^{J}(X)|^2\delta _{J}}{\sum\limits_{J}|P^{J}(X)|^2}.
\end{equation}   
Here $J$ goes through $k$-element subsets of $\{1,\ldots ,n\}$ and $\delta _{J}\in \R ^{n}$ is given by
$(\delta _{J})_{i}=1$ for $i\in J$, while $(\delta _{J})_{i}=0$ for $i\notin J$.

The moment map is equivariant under the action of compact torus meaning that $\mu (T^n\cdot X) = \mu (X)$. It easily follows since $P^{J}(T^n\cdot X) = t_{j_1}\cdots t_{j_k}\cdot P^{J}(X)$, where $J=\{j_1,\ldots ,j_k\}$. The image of the moment map is hypersimplex $\Delta_{n,k}$ which can be defined as the set of all points  $(x_1,\ldots ,x_n)\in \R ^n$ such that  $0\leq x_i\leq 1$ and $\sum _{i=1}^{k}x_i=k$.

By the  classical convexity theorem of~\cite{AT} and~\cite{GUST} it follows that  $\mu (\overline{\OO _{\C}(X)})$  is a convex polytope for any $X\in G_{n,k}$.  More precisely the following result  holds:
\begin{thm}\label{moment-map-polytope}
Let $\OO _{\C}(X)$ be an orbit of an element $X\in G_{n,k}$ under the canonical action of $(\C ^{*})^{n}$. Then
$\mu (\overline{\OO _{\C}(X)})$ is a convex polytope in $\R ^{n}$ whose vertex set is given by $\{ \delta _{J} | P^{J}(X)\neq 0 \}.$ The mapping $\mu$ gives a bijection between $p$-dimensional orbits of the group $(\C ^{*})^{n}$ in 
$\overline{\OO _{\C}(X)}$ and $p$-dimensional open faces of the polytope $\mu (\overline{\OO _{\C}(X)})$.
\end{thm} 

\begin{rem} We want to recall that the convexity theorem states that the image of the moment map is convex hull of the images of fixed points for  a Hamiltonian torus action on compact, connected, symplectic manifold. The example of such action  is the standard action of $T^n$ on $\C P ^n$ endowed with the K\"ahler structure given by the Fubini-Study metric and with the standard moment map $\mu : \C P^n \to \R^ {n+1}$ given by 
\[
\mu (z) = \frac{1}{\|z\| ^2}(|z_1|^2(1,0,\ldots ,0)+\cdots +|z_{n+1}|^2(0,\ldots,0,1)).
\]
The composition of this moment map together with Pl\"ucker embedding of the Grassmann manifold into complex projective space gives the moment map~\eqref{moment-map}.  This follows from the following more general result.

Namely,  according to~\cite{Kir}, if a compact Lie group $G$ acts on $\C P^n$ via homomorphism $\psi : G\to U(n+1)$  and  $M \subseteq \C P^n$ is a nonsingular subvariety invariant under the action of $G$, than the restriction of Fubini-Study metric on $\C P^n$ gives $M$ a K\"ahler structure that is preserved by $G$. Moreover, a moment map
$\mu : M^{2n}\to \gg^{*}$, where $\gg$ denotes the Lie algebra for $G$, is defined by
\[
\mu (x)(v) = \frac{1}{2\pi i\|x^{*}\|^{2}}(\bar{x}^{*})^{T}\psi _{*}(v)x^{*},
\]
where $v\in \gg$ and $x^{*}$ is non-zero vector in $\C ^{n+1}$ lying over the point $x\in M\subseteq \C P^n$.
\end{rem}    

\section{The Strata on Grassmanians}
In recalling the notion of the strata on Grassmanians we follow~\cite{GELSERG},~\cite{GEL4}.
We provide three definitions of the strata on Grassmannians for which it is proved in~\cite{GELSERG},~\cite{GEL4} to be equivalent.
\begin{enumerate}
\item For the standard orthonormal basis $e_1,\ldots ,e_n$ in $\C ^{n}$ and $k$ - dimensional complex subspace $X\in G_{n,k}$ there are $n$ vectors $\pi _{X}(e_1),\ldots ,\pi _{X}(e_n)$ determined by the projection $\pi _{X} : \C ^n\to \C ^{n}/X\cong \C^{n-k}$.
Then it is defined rank function on subsets $J\subset \{ 1,\ldots ,n\}$ by
\[
 \rank (J) = \dim _{\C }(\text{span} \{ \pi _{X}(e_j) | j\in J \} ).
\]
Two points $X_1,X_2\in G_{n,k}$ are said to belong to the same Grassmann stratum $W$ of $G_{n,k}$ if for each subset
$J\subset \{ 1,\ldots ,n\}$ it holds
\[
\dim _{\C }(\text{span} \{ \pi _{X_1}(e_j) | j\in J \} ) = \dim _{\C }(\text{span} \{ \pi _{X_2}(e_j) | j\in J \} ).
\]
\item
Two points $X_1,X_2\in G_{n,k}$ are said to belong to the same Grassmann stratum $W$ of $G_{n,k}$ if 
\[
\mu (\overline{\OO _{\C}(X_1) }) = \mu (\overline{\OO _{\C}(X_2) }).
\]
\item
It is well known the notion of the Schubert  cell decompositions for the Grassmann manifolds. These decompositions are parametrized by the complete flags in $\C ^n$. More precisely, to  the standard orthonormal basis $e_1,\ldots ,e_n$ in $\C ^{n}$ it corresponds the flag $\C ^1\subset \C^2 \ldots \subset \C^{n}$, where $\C ^{i} = span\{e_1,\ldots ,e_i\}$. The Schubert cell decomposition for the Grassmannain $G_{n,k}$ corresponding to  this flag is defined as follows: for the Schubert symbol $(i_1,\ldots ,i_k)$, where $1\leq i_1<i_2\ldots <i_k\leq n$, the Schubert cell is defined by
\[
C_{(i_1,\ldots ,i_k)}=\{ X\in G_{n,k} | \dim X\cap \C ^{i_j}= j,\; \dim X\cap \C ^{i_j-1}<j \}.
\]
For any permutation $\sigma \in S_{n}$ one can consider the new orthonormal basis $e_{\sigma (1)},\ldots ,e_{\sigma (n)}$ in $\C ^n$ and the corresponding flag $\C ^{1}_{\sigma (1)}\subset \C ^{2}_{\sigma (2)}\ldots \subset \C ^{n}_{\sigma (n)}$,
where $\C ^{i}_{\sigma} = span\{e_{\sigma (1)},\ldots ,e_{\sigma (i)}\}$. It also gives the Schubert cell decomposition by
\[
C^{\sigma}_{(i_1,\ldots ,i_k)} = \{ X\in G_{n,k} | \dim X\cap \C ^{\sigma (i_j)} = j,\; \dim X\cap \C ^{\sigma (i_j)-1}<1 \}.
\]
The refinement of all Schubert cell decompositions $\{ C^{\sigma} \}$, $\sigma \in S_{n}$ gives also the cell decomposition of $G_{n,k}$. Each cell in such decomposition is obtained as the intersection of $n!$ cells picked up from the each of $n!$  Schubert cell decompositions. 

Two points $X_1,X_2\in G_{n,k}$ are said to belong to the same Grassmann stratum of $G_{n,k}$ if they belong to the same  cell obtained by this refinement.
\end{enumerate}
 
\section{Polytopes for the Grassmannian $G_{4,2}$}  
We describe the polytopes that are the images of the closure of $(\C ^{*})^{4}$-orbits  on $G_{4,2}$ by the moment map. Such polytopes we call admissible polytopes. The set of vertices of any such polytope
is by~\eqref{moment-map} and Theorem~\ref{moment-map-polytope}  subset of the set of the following points: $\delta_{12}=(1,1,0,0), \delta _{13}=(1,0,1,0), \delta _{14} =(1,0,,0,1), \delta_{23}=(0,1,1,0), \delta _{24}=(0,1,0,1), \delta _{34}=(0,0,1,1)$. 
\begin{lem}\label{polytopes}
The convex polytopes that are the images of the closure of $(\C ^{*})^{4}$-orbits  on $G_{4,2}$ by the moment map have the following sets of vertices:
\begin{enumerate}
\item six-point set - $\delta _{12}, \delta _{13}, \delta _{14}, \delta _{23}, \delta _{24}, \delta _{34}$; 
\item  each five-point subset  of these six points; 
\item  the following four-point sets -- $\delta _{13}, \delta _{14}, \delta _{23}, \delta _{24}$; $\delta _{12}, \delta _{14}, \delta _{23}, \delta _{34}$; $\delta _{12}, \delta _{13}, \delta _{24}, \delta _{34}$;
\item  the following  three-point sets -- $\delta _{12}, \delta _{13}, \delta _{14}$; $\delta _{12}, \delta _{13}, \delta _{23}$; $\delta _{13}, \delta _{14}, \delta _{34}$; $\delta _{13}, \delta _{23}, \delta _{34}$; $\delta _{12}, \delta _{14}, \delta _{24}$;
$\delta _{12}, \delta _{23}, \delta _{24}$; $\delta _{14}, \delta _{24}, \delta _{34}$; $\delta _{23}, \delta _{24},\delta _{34}$;
\item  one of two-point sets -- $\delta _{12},\delta _{13}$; $\delta _{12}, \delta _{14}$; $\delta _{12}, \delta _{23}$; $\delta _{12}, \delta _{24}$; $\delta _{13}, \delta _{14}$; $\delta _{13}, \delta _{23}$; $\delta _{13}, \delta _{34}$; $\delta _{14}, \delta _{24}$; $\delta _{14}, \delta _{34}$; $\delta _{23}, \delta _{24}$; $\delta _{23}, \delta _{34}$; $\delta _{24}, \delta _{34}$.
\item each point form the set  $\{\delta _{12},\ldots ,\delta_{34} \}$.
\end{enumerate}
\end{lem}
\begin{figure}[h]
\begin{center}
\includegraphics[width=5cm]{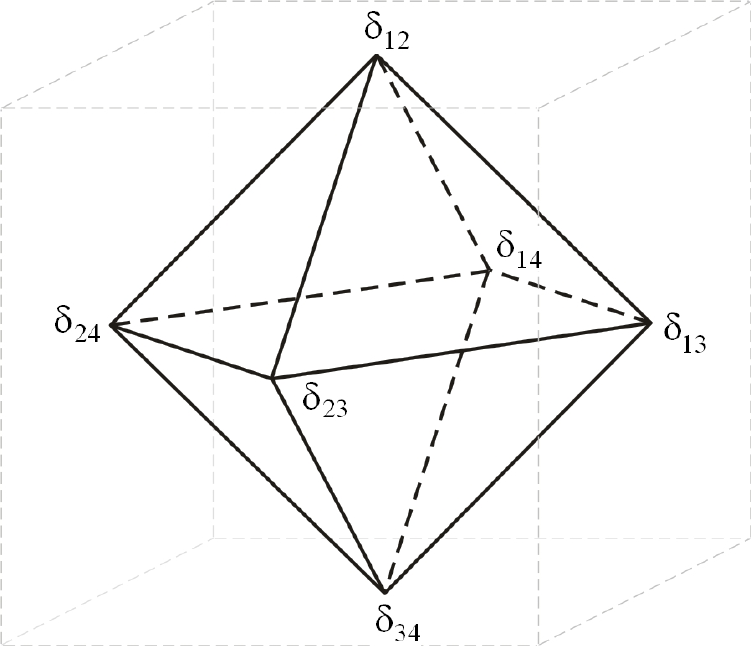}  

\end{center}
\end{figure}

\begin{proof}
We use Theorem~\ref{moment-map-polytope} and for each  polytope spanned by the vertices given by the enumerated cases provide explicitly a point $X\in G_{4,2}$ such that the closure of its $(C^{*})^4$-orbit  maps to it by the moment map.

For the first case consider the point $X\in G_{4,2}$ given in standard basis by the matrix
\[
A_{X}=\left(\begin{array}{cc}
1 & 0\\
-1 & 1\\
1 & 1\\
0 & 1 
\end{array}\right) .
\]
Then all $2\times 2$ minors of $A_X$ are non-trivial implying that $\mu (\overline{\OO _{\C}(X)})$ is a convex polytope
spanned by the vertices $\delta _{12},\ldots ,\delta _{34}$. 

The polytopes with five vertices $\{\delta_{12},\ldots ,\delta_{34}\} - \delta_{ij}$ are realised by the closure of $(\C ^{*})^4$-orbits of the elements $X_{ij}\in G_{4,2}$, where
$1\leq i<j\leq 4$ and which are represented by the matrices 
\[
A_{X_{12}}=\left(\begin{array}{cc}
1 & 1\\
1 & 1\\
1 & 0\\
0 & 1 
\end{array}\right) , \;\;
A_{X_{13}}=\left(\begin{array}{cc}
1 & 1\\
1 & 0\\
1 & 1\\
0 & 1 
\end{array}\right) , \;\;
A_{X_{14}}=\left(\begin{array}{cc}
1 & 1\\
1 & 0\\
0 & 1\\
1 & 1 
\end{array}\right) ,
\]
\[
A_{X_{23}}=\left(\begin{array}{cc}
1 & 0\\
1 & 1\\
1 & 1\\
0 & 1 
\end{array}\right) , \;\;
A_{X_{24}}=\left(\begin{array}{cc}
1 & 0\\
0 & 1\\
1 & 1\\
0 & 1 
\end{array}\right) , \;\;
A_{X_{34}}=\left(\begin{array}{cc}
1 & 0\\
0 & 1\\
1 & 1\\
1 & 1 
\end{array}\right) .
\]

In order to decide which four-point subsets appear as the vertex set for polytopes obtained as the images of the moment map, we note the following. If two Pl\"ucker coordinates with a common index of a point $X\in G_{4,2}$ are zero than one more Pl\"ucker coordinate has to be zero as well. It implies that the image of the closure of $(\C ^{*})^4$-orbit for $X\in G_{4,2}$ has four vertices in the case when two Pl\"ucker coordinates for $X$ with no common index are zero. We obtain the following:
\[
A_{X}=\left(\begin{array}{cc}
1 & 0\\
1 & 0\\
0 & 1\\
0 & 1 
\end{array}\right) , \;\; P^{12}(X)=P^{34}(X)=0, \;\; \text{vertex set:}\;\; \delta_{13}, \delta_{14}, \delta_{23}, \delta _{24};   
\]
\[
A_{X}=\left(\begin{array}{cc}
1 & 0\\
0 & 1\\
1 & 0\\
0 & 1 
\end{array}\right) , \;\; P^{13}(X)=P^{24}(X)=0, \;\; \text{vertex set:}\;\; \delta_{12}, \delta _{14}, \delta _{23}, \delta _{34};   
\]
\[
A_{X}=\left(\begin{array}{cc}
1 & 0\\
0 & 1\\
0 & 1\\
1 & 0 
\end{array}\right) , \;\; P^{14}(X)=P^{23}(X)=0, \;\; \text{vertex set:}\;\; \delta _{12}, \delta _{13}, \delta _{24}, \delta _{34}.   
\]

To describe the triangles that can be obtained as the images of the moment map for  
$\overline{\OO _{\C}(X)}$ we
note that $X$ must have exactly three zero Pl\"ucker coordinates. It may happen  when  the two Pl\"ucker coordinates for $X$ with the common index are zero or when the matrix $A_{X}$ has zero row. The possible examples for such a matrix are   
\[
A_{X}=\left(\begin{array}{cc}
1 & 0\\
1 & 0\\
1 & 0\\
0 & 1 
\end{array}\right) ,\;\;    
A_{X}=\left(\begin{array}{cc}
1 & 0\\
1 & 0\\
0 & 1\\
1 & 0 
\end{array}\right) , \;\;    
A_{X}=\left(\begin{array}{cc}
0 & 1\\
1 & 0\\
0 & 1\\
0 & 1 
\end{array}\right) , \;\;
A_{X}=\left(\begin{array}{cc}
0 & 1\\
1 & 0\\
1 & 0\\
1 & 0 
\end{array}\right) ,
\] 
\[
A_{X}=\left(\begin{array}{cc}
0 & 0\\
0 & 1\\
1 & 0\\
1 & 1 
\end{array}\right) , \;\;    
A_{X}=\left(\begin{array}{cc}
1 & 0\\
0 & 0\\
0 & 1\\
1 & 1 
\end{array}\right) , \;\;
A_{X}=\left(\begin{array}{cc}
1 & 0\\
0 & 1\\
0 & 0\\
1 & 1 
\end{array}\right) , \;\;  
A_{X}=\left(\begin{array}{cc}
1 & 0\\
0 & 1\\
1 & 1\\
0 & 0 
\end{array}\right) .   
\]
The corresponding polytopes we obtain by the moment map are triangles with the set of vertices $\delta_{14},\delta_{24},\delta_{34}$; $\delta _{13}, \delta _{23},\delta _{34} $; $\delta _{12},\delta _{23},\delta _{24}$; $\delta _{12},\delta _{13},\delta _{14}$; $\delta _{23}, \delta _{24}, \delta _{34}$; $\delta _{13}, \delta _{14}, \delta _{34}$; $\delta _{12}, \delta _{14}, \delta _{24}$; $\delta _{12}, \delta _{13}, \delta _{23}$ respectively. 

To finish to proof we  note that  if  two vertices correspond to the Pl\"ucker coordinates with complementary indexes that one can not  find $X\in G_{4,2}$ such that closure of its orbit maps to the interval
defined by the given vertices. For the other pairs of vertices obviously it is possible.
\end{proof}

\subsection{Geometric description of the polytopes.} 
We derive form Lemma~\ref{polytopes} the following table in which the first row gives the dimensions of the  polytopes  which are obtained as the image of the orbits on $G_{4,2}$ by the  moment map, while the second row gives their number in corresponding dimension:
\begin{equation}\label{number}
\left[\begin{array}{cccc}
3 & 2 & 1 & 0\\
7 & 11 & 12 & 6 
\end{array}\right].
\end{equation}
We explain these polytopes in more detail. Note that the  points $\delta _{12},\ldots \delta _{34}$ belong to the hyperplane in $\R ^4$ given by the equation $x_1+x_2+x_3+x_4=2$. Therefore any such polytope might be at most three-dimensional. Moreover each of four vertices which by Lemma~\ref{polytopes} span the polytope belong to the same plane in this hyperplane.  Now Lemma~\ref{polytopes} gives that in the image of the moment map we have the polytope $\Delta _{4,2}$ which is convex hull of the vertices $\delta_{12},\ldots ,\delta _{34}$.  Note that to the interior of this polytope maps the orbit of any element $X\in G_{4,2}$ whose all Pl\"ucker coordinates are non-zero. The other $6$  three-dimensional polytopes are the pyramids having one of the three diagonal rectangles as the basis. Among two-dimensional  polytopes there are $8$ triangles forming the boundary of the octahedron and three diagonal rectangles inside of an octahedron. The one-dimensional polytopes are those from the boundary from the octahedron.
\begin{rem}
Note that not  every polytope spanned by some subset  of  vertices for $\Delta _{4,2}$ can be realised as the polytope of the moment map. More precisely, except for the three squares and six pyrhamids, none of the polytope which intersect the interior of the octahedron can not be obtained as the image of the moment map for some orbit on $G_{4,2}$.
\end{rem}
\begin{rem}
The  admissible polytopes in $\Delta _{4,2}$ of dimension three  are not simple polytopes, while all polytopes in dimension $\leq 2$  are simple.
\end{rem}   
\subsection{Lattice of the strata on $G_{4,2}$}
As we see from its definition each stratum on Grassmannian maps by  the moment map to the interior of exactly one polytope which we  will call the polytope of the stratum. In this context we will use the phrase a stratum over the polytope. There is a bijection between all polytopes in the image of the moment map and all strata on  Grassmannian.  In this way the  table~\eqref{number} shows that stratification of  $G_{4,2}$ consists of  $36$ strata  corresponding to the polytopes
in the image of the moment map. For example, there are $7$ strata that map to the interior of three-dimensional polytopes and they
can be described as follows:
\[
W_{\Delta _{4,2}} = \{ X\in G_{4,2}\; |\; P^{J}(X)\neq 0\},
\]
\[
W_{P} = \{ X\in G_{4,2}\; |\; P^{j_1j_2}(X)=0,\; P^{J}(X)\neq 0,\; J\neq j_1j_2 \},
\]
where $\dim P=3$ and $P\neq \Delta _{4,2}$. Here $1\leq j_1<j_2\leq 4$ and $J\subset \{ 1,2,3,4 \},|J|=2$.

We can introduce on the set of strata for $G_{4,2}$ a lattice structure by saying $W_{P_1}\leq W_{P_2}$ if and only if $\mu (W_{P_1})\subseteq \mu (W_{P_2})$. In order to have a lattice structure we assume here that empty set is a strata as well. We describe the graph of this lattice. At the bottom and the top of this lattice are the empty strata and the strata that maps to $\Delta _{4,2}$. The first level of the lattice corresponds to the  point-strata and at each point we have four arrows going to the "one" - strata that map to the intervals. From each such stratum go three arrows to the "two" - strata that map to the two-dimensional polytopes, while from each "two" - strata we have two arrows going to "three" - strata which do not map to $\Delta _{4,2}$. From each of these strata we have further arrow to the strata that map to $\Delta _{4,2}$.     

\section{The structure of the stratification for $G_{4,2}$}\label{structure}
We use the results from previous sections to describe the topological structure of the orbit space $G_{4,2}/T^4$.  

We first note that using Pl\" ucker coordinates there is on  $G_{4,2}$ canonical atlas consisting of six charts: for any $J\subset \{ 1,2,3,4 \}$, $|J|=2$, we have a chart $M_{J}=\{ X\in G_{4,2}\; |\; P^{J}(X)\neq 0\}$. The homeomorphism $u_{J} : M_{J}\to \C ^4$ is given as follows: for $X\in J$ we can choose the basis such that in the matrix $A_{X}$, which represents $X$,  the submatrix determined by the rows indexed by $J$ is the identity matrix. Then $u_{J}(X) = (a_{ij}(X))\in \C^ 4$, where $i\notin J$. Note the following:
\begin{itemize}
\item the charts $M_{J}$ are invariant under the action of $(\C ^{*})^4$,
\item this action, by the homeomorphism $u_{J}$, induces the action of $(\C ^{*})^4$ on $\C ^4$.
\end{itemize}
The orbits of this induced action in the charts  can be explicitly described.
We want to note that the Weyl group which acts on Grassmann manifolds permutes these charts, so it is enough to describe the orbits in  one chart and then apply permutations. 

\begin{prop}\label{orbits}
The six-dimensional orbit of a point $(a_1,a_2,a_3,a_4)\in \C ^4$ for   $(\C ^{*})^4$ - action on $\C ^4$  which corresponds to the  chart $(M_{J}, u_{J})$ on $G_{4,2}$ is:
\begin{itemize}
\item the hypersurface in $\C ^4$ given by the equation
\[
\frac{z_1z_4}{z_2z_3} = c,\;\; c\neq 0 - \text{constant},
\]
if all $a_i$ are non-zero;
\item $(\C^{*})^{3}$ in the case when exactly one coordinates  $a_i$ is zero, $(\C ^{*})^2$ when exactly two coordinates $a_i$ are zero, $\C ^{*}$ when exactly three  coordinates $a_i$ are zero and  $O=(0,0,0,0)$ is the fixed point.
\end{itemize}
\end{prop}
\begin{proof}
We provide the proof for the action induced by the chart $(M_{12}, u_{12})$. For $X\in M_{12}$ we can choose the basis such that  it is represented by the matrix
\[
A_{X}=\left(\begin{array}{cc}
1 & 0\\
0 & 1\\
a_1 & a_3\\
a_2 & a_4 
\end{array}\right) .   
\]
Than the orbit $\O _{\C}(X)$ has the form
\[
\left(\begin{array}{cc}
t_1 & 0\\
0 & t_2\\
t_3a_1 & t_3a_3\\
t_4a_2 & t_4a_4 
\end{array}\right)
=\left(\begin{array}{cc}
1 & 0\\
0 & 1\\
\frac{t_3}{t_1}a_1 & \frac{t_3}{t_2}a_3\\
\frac{t_4}{t_1}a_2 & \frac{t_4}{t_2}a_4 
\end{array}\right) ,   
\]
where $(t_1,t_2,t_3,t_4)\in (\C ^{*})^{4}$.
In other words we have an action of $(\C ^{*})^4$ on $\C ^4$ via the representation $(t_1,t_2,t_3, t_4)\to (\frac{t_3}{t_1},\frac{t_4}{t_1},\frac{t_3}{t_2},\frac{t_4}{t_2})$.
If put $\tilde{t_1}=\frac{t_3}{t_1}$, $\tilde{t_2}=\frac{t_4}{t_1}$, $\tilde{t_3}=\frac{t_3}{t_2}$ and $\tilde{t_4}=\frac{t_4}{t_2}$, we obtain that  this action induces an effective action of $(\C ^{*})^3$ on $\C ^4$. Then the orbit of an element $(a_1,a_2,a_3,a_4)$ can be written in the form  $(t_1a_1,t_2a_2,t_3a_3,  (t_3t_2/t_1)a_4)$, where   we denoted  $\tilde{t_i}$ by $t_i$ as well.
Therefore the  orbit will be  $(\C ^{*})^3$, $(\C ^{*})^2$ or $\C ^{*}$  if exactly one, two or three  of $a_i$'s equal to zero respectively. If $a_i\neq 0$, $i=1,\ldots ,4$ then  $(z_1,z_2,z_3,z_4)\in \C ^4$ is in this orbit if and only  if it satisfies the equation
\[
\frac{z_1z_4}{z_2z_3} = c \;\; \text{for}\;\;  
c=\frac{a_1a_4}{a_2a_3}.
\]
\end{proof}

\begin{thm}\label{orbits_chart}
In each chart $(M_{J}, u_{J})$ on Grassmannian $G_{4,2}$ there are 
\begin{enumerate}
\item four orbits homeomorphic by $u_{J}$ to $(\C ^{*})^3\subset \C ^4$ and they all belong to different strata mapping to the three-dimensional polytopes different from $\Delta _{4,2}$;
\item six orbits homeomorphic by $u_{J}$ to $(\C ^{*})^2\subset \C ^{4}$ and they all belong to different strata;
\item four orbits homeomorphic by $u_{J}$ to $\C ^{*}\subset \C ^{4}$ and they all belong to different strata;
\item Orbits homeomorphic by $u_{J}$ to the  hypersurfaces in $\C ^{4}$ given by the equation $\frac{z_1z_4}{z_2z_3}=c$,
$c\neq 0$, where:
\begin{itemize}
\item for $c\neq 1$ they all belong to the same strata which maps to octahedron $\Delta _{4,2}$;
\item the orbit given by $c=1$  belongs to the strata mapping to the three-dimensional polytope which does not contain the vertex  $delta_{L}$, where $L = \{1,2,3,4\}-J$.
\end{itemize}
\end{enumerate}
\end{thm}
\begin{proof}
Note that all polytopes which correspond to the orbits  from $M_{J}$ contain the vertex $\delta _{J}$.
We describe these orbits in one chart and by the action of the Weyl group the same  holdi for the other charts as well. The orbits of the first type we obtain if exactly one of the entries $a_i$'s of the matrix $A_{X}$ is zero what implies 
that $X$ has exactly one zero Pl\" ucker coordinate. Moreover all such $X$ having the fixed entry $a_i$ equal to zero belong to the same orbit. Therefore there are four such orbits which map by the moment map to different  three-dimensional polytopes which are also different from $\Delta _{4,2}$. Note that all such polytopes contain the vertex  $\delta _{L}$, $L=\{1,2,3,4\}-J$ as well. Analogously we prove the second and the third statement. For the fourth statement it directly checks that all  Pl\"ucker coordinates for the  orbits given by the surfaces $\frac{z_1z_4}{z_2z_3}=c$, $c\neq 0$ are non-zero and therefore, all such orbits belong to the same strata
which maps to $\Delta _{4,2}$ by the moment map. The orbit given by the hypersurface $\frac{z_1z_4}{z_2z_3}=1$ has the  Pl\"ucker coordinate $P^{L}=0$, $L=\{1,2,3,4\}-J$ implying that it maps to the three-dimensional polytope which does not contain the vertex $\delta _{L}$. Note that this polytope is different from the polytopes that correspond to the orbits from the first statement of the Theorem.      
\end{proof}
\begin{rem}
In the case when we want to differentiate  the orbits  which belong to the coordinate subspaces we introduce notation  $\C ^{*}_{I} = \{(z_1,z_2,z_3,z_4) \in \C ^4, z_i =0$ iff $i\notin I\}$,   where $I \subset \{1,\ldots, 4\}$ and $\| I\| \leq 3$. Note that the admissible polytopes for $\C ^{*} _{I}$, where $I=\{1,2\},\{1,3\},\{2,4\},\{3,4\}$ are triangles while the admissible polytopes for $\C ^{*}_{I}$ where $I=\{1,4\},\{2,3\}$ are squares.
\end{rem}
From Theorem~\ref{orbits_chart} and its proof it directly follows:
\begin{cor}\label{strata_chart}
In each chart $(M_{J}, u_{J})$ on $G_{4,2}$  a stratum over a polytope different from $\Delta _{4,2}$ consists of one orbit. 
\end{cor}
It shows that the same is true for such strata on the whole Grassmannian $G_{4,2}$:
\begin{cor}\label{Gr_other_strata}
On Grassmannian $G_{4,2}$ any stratum over a polytope different form $\Delta _{4,2}$ consists of one orbit.
\end{cor}  
\begin{proof}
Consider the stratum $W_{P}$ over some polytope $P$ different from $\Delta _{4,2}$. This polytope is determined by the vertices which correspond to the non-zero Pl\"ucker coordinates for $X$, where $X$ is an arbitrary element in $W_{P}$. On the other hand, the charts
on $G_{4,2}$ that contain any such $X$ are indexed by the non-zero Pl\" u cker coordinates for $X$. It implies that
\[
W_{P} \subset \bigcap M_{J},\; P^{J}(X)\neq 0,\; X\in W_{P} .
\]
It follows from Corollary~\ref{strata_chart}  that $W_{P}$ consists of one orbit.   
\end{proof}
As for the remaining stratum we prove:
\begin{cor}
The stratum $W_{\Delta _{4,2}}$ is open everywhere dense set in $G_{4,2}$ of dimension $8$. It can be obtained  as the intersection of all charts $M_{J}$ on $G_{4,2}$. 

\end{cor}
\begin{proof}
It follows from  Theorem~\ref{orbits_chart} and Corollary~\ref{Gr_other_strata} that the dimension of any strata on $G_{4,2}$ different from $W_{\Delta _{4,2}}$ is $\leq 6$. Since $G_{4,2}$ is a compact manifold of dimension eight we derive that this stratum is everywhere dense set of dimension eight. Since $\mu (W_{\Delta _{4,2}}) = \stackrel{\circ}{\Delta _{4,2}}$ it follows that  all  Pl\" ucker coordinates for any $X\in W_{\Delta _{4,2}}$  are non-zero. Therefore $W_{\Delta _{4,2}}\subset M_{J}$ for any chart $M_J$ on $G_{4,2}$.   
\end{proof} 

We obtain complete description of the stratification for $G_{4,2}$.
\begin{cor}
The first row of the following table gives the dimension of the strata on Grassmannian  $G_{4,2}$, while the second row gives their number in the corresponding dimension:
\begin{equation}
\left[\begin{array}{ccccc}
8 & 6 & 4 & 2 & 0\\
1 & 6 & 11 & 12 & 6
\end{array}\right].
\end{equation}
Moreover each strata of dimension $\leq 6$ consists of one orbit of the corresponding dimension.
\end{cor}

\section{The singularities of toric varieties obtained as compactification of $(\C ^{*})^4$ - orbits on $G_{4,2}$}

We describe first  the closure of the orbits in the charts.
\begin{prop}
The boundary of an orbit for the action of  $(\C ^{*})^{3}$ on $\C ^{4}$  induced by  the  chart $(M_{J}, u_{J})$ is as follows:
\begin{itemize}
\item for the orbit $\C _{I}^{*}$, $\| I\| =1$  it is the point $O=(0,0,0,0)$;
\item for the orbit $\C _{I}^{*}$, where $I = \{i,j\}$ it is $\C _{I_1}^{*}\cup \C_{I_2}^{*}\cup O$, where $I = \{i,j\}$, $I_1=\{i\}$ and $I_{2} = \{j\}$;
\item for the orbit $\C _{I}^{*}$, where $I =\{i, j, k\}$   it is $\C _{I_1}^{*} \cup \C _{I_2}^{*}\cup \C _{I_3}^{*}\cup \C _{I_4}^{*}\cup \C _{I_5}^{*}\cup \C _{I_6}^{*}\cup O$, where $I_1=\{i,j\}, I_{2}=\{i,k\}, I_{3}=\{j,k\}, I_4 = \{i\}, I_{5}=\{j\}, I_6=\{k\}$;  
\item for the orbits $z_2z_3=cz_1z_4$ it is $\C_{I_1}^{*}\cup \C _{I_2}^{*}\cup \C_{I_3}^{*}\cup \C _{I_4}^{*}\cup \C _{I_5}^{*}\cup \C_{I_6}^{*}\cup \C_{I_7}^{*}\cup \C _{I_8}^{*}\cup O$, where $I_1=\{1,2\}, I_{2}=\{1,3\}, I_3=\{2,4\}, I_{4}=\{3,4\}, I_5=\{1\}, I_6=\{2\}, I_7=\{3\}, I_{8}=\{4\}$.
\end{itemize}
\end{prop}
\begin{proof}
The boundary  of any orbit in $\C ^4$ given by the chart $(M_{J}, u_{J})$ is  obtained when some of the parameters $t_1,t_2,t_3$ tends to zero. Without loss of generality we work in the chart $(M_{12}, u_{12})$ and   use the description of the orbits from Proposition~\ref{orbits}. For the orbits of the type $\C_{I}^{*}$, $\|I\|\leq 3$ it is clear that their boundary is the union of their coordinate subspaces of  less dimensions 
as it is given in the statement. 
For the orbit which is given as the hypersurface we see that its closure consists of the following points:
\[
(z_1,z_2,0,0), \; t_3\rightarrow 0;\; (z_1,0,z_3,0), \; t_2\rightarrow 0; 
\]
\[(z_1,0,0,0), \; t_2, t_3\rightarrow 0; (0,z_2,0,0), t_1,t_3\rightarrow 0, \; t_{3}=\o (t_1); \; (0,0,z_3,0), \; t_1, t_2\rightarrow 0,\; t_{2}=\o (t_1);
\]
\[
(0,0,0, 0), \; t_1,t_2, t_2\rightarrow 0, t_2t_3=\o (t_1);
\]
\[
(0,z_2,0,z_4), \; t_1, t_3\rightarrow 0, t_3 = O(t_1), \;  (0, 0, z_3, z_4), \; t_1, t_2\rightarrow 0, \; t_2 = O(t_1);
\]
\[
(0,0,0,z_4), \; t_1, t_2, t_3\rightarrow 0, t_2t_3 = O(t_1).
\]
\end{proof}  

\begin{prop}
In an arbitrary chart $(M_{J}, u_{J})$   the closure of the orbits of the type $\C _{I}^{*}$, $\| I\| \leq 3$ are smooth manifolds. The closure of an orbit given as $z_2z_3=cz_1z_4$, $c\neq 0$, has singularity at the point $O=(0,0,0,0)$.
\end{prop}
\begin{proof}
The closure of an orbit of the type $\C _{I} ^{*}$, $\| I\| \leq 3$ is $\C _{I}$, so it is a smooth manifold.
The singularities for the closure of the orbits $z_2z_3=cz_1z_4$ may appear only on its boundary.  The points  that belong to the two-dimensional orbits $\C _{I}^{*}$, $\|I\|=2$ of its boundary, do not belong to the closure of any other orbit, so these points are regular. The points that belong to one-dimensional orbits $\C _{I}^{*}$, $\|I \|=1$ of its boundary, say to the orbit when $I=\{1\}$,  belong to the closure of exactly two orbits, $\C _{I_1}^{*}$ and $\C _{I_2}^{*}$, where $I_1=\{1,2\}$ or $I_2=\{1,3 \}$. Since the dimension of our hypersurface orbit is six, we deduce that these points are regular.
As for the point $O=(0,0,0,0)$, it belongs to the closure of four one-dimensional orbits from the boundary of this
hypersurface orbit. These orbits intersects transversally in $\C ^4$ at the point $O$. Since the dimension of the hypersurface orbit is six, it implies that the point $O$ is a singular point of its closure. 
\end{proof}

Recall that given a smooth toric manifold, its image by the moment map is a simple convex polytope.
Since for any $X\in G_{4,2}$ an orbit $\overline{\OO _{\C}(X)}$ is a toric manifold,  if it is  a smooth manifolds or equivalently if  it does not have any singularities, then $\mu (\overline{\OO _{\C}(X)})$ is a simple polytope. This gives  an explicit  description of the orbits on $G_{4,2}$.
\begin{thm}\label{closure-orbits}
The closure of $(\C ^{*})^4$-orbit  of  an element $X\in G_{4,2}$ different from a fixed point is:  
\begin{enumerate}
\item A six-dimensional toric manifold with six  singular points if $\mu (\overline{\OO _{\C}(X)}) =\Delta _{4,2}$;
\item A six-dimensional toric manifold with  one singular point if  $\mu (\overline{\OO _{\C}(X)}) =P$, for admissible polytope $P\neq \Delta _{4,2}$ such that $\dim P =3$;
\item  $\C P^2$ if $\mu (\overline{\OO _{\C}(X)})$ is a triangle;
\item $\C P^1\times \C P^1$ if $\mu (\overline{\OO _{\C} (X)})$ is a square;
\item $\C P^1$ if $\mu (\overline{\OO _{\C} (X)})$ is an interval;
\end{enumerate}
\end{thm}
\begin{proof}
Since an admissible three-dimensional polytope for $G_{4,2}$ is not simple it follows that any six-dimensional orbit must have singularities. The orbit of an element $X\in G_{4,2}$ such that $\mu (\overline{\OO _{\C}(X))} =\Delta _{4,2}$
belongs to any chart $(M_{J}, u_{J})$. These means that in any chart the point $O=(0,0,0, 0)$ is the unique singular point of its closure. In the chart $(M_{J}, u_{J})$ to the point $O$  corresponds an element $X_{J}\in M_{J}$ such that the only non-zero Pl\"ucker coordinates for $X_{J}$ is $P^{J}(X_{J})$ meaning that  $X_{J}$ is a fixed point. Therefore the closure of an orbit for such $X$ has exactly six singular points given by the fixed points. In the same way we prove the second statement. Any toric manifold whose corresponding polytope is  a triangle or  intervals "equivariantly diffeomorphic" to  $\C P^2$ or $\C P^1$ respectively.  It is also known  that a quasitoric manifold whose corresponding polytope  is  a square is "equivariantly diffeomorphic" to $\C P^1 \times \C P^1$, $\C P^2\# \C P^2$ or $\C P^2\# \overline{ \C P^2}$. Note that actually we do not need to consider $\C P^2\# \C P^2$ since we look for a toric manifolds. 
These three cases are differentiate by the value of the characteristic function on the boundary of the square.
Let $P$ be a square with the vertices $\delta _{12}$, $\delta _{14}$, $\delta _{23}$ and $\delta _{34}$. The    $(\C ^{*})^{4}$-orbit  which maps to this square is two-dimensional and it admits the free action of $T^2\subset T^6$. It is in the chart $M_{12}$   an orbit of an element given by a matrix for which $a_{11}=a_{22}=1$, $a_{12}=a_{21}=0$ and its coordinates are $a_{31},a_{42}\neq 0$ while $a_{32}=a_{41}=0$ . The torus $T^2 = \{(t_3,t_4) | t_3,t_4\in S^1\}= \{(1,1,t_3,t_4) | t_3,t_4\in S^1\}\subset T^4$ acts freely on this orbit and its action is  given by $(t_{3}a_{31}, 0,0,t_4a_{42})$. It implies that  the stabilizers of the orbits which map to the boundary intervals of $P$ which contain the vertex $\delta _{12}$ are $(0,1)$ and $(1,0)$. For the other two intervals  we consider the chart $M_{13}$. In this chart  the $(\C ^{*})^4$-orbit is given by a matrix for which $a_{12}=a_{21}=a_{32}=a_{41}=0$, $a_{31}=a_{42}=1$ and $a_{11},a_{22}\neq 0$. The action of the torus $T^2$ is given by $(\frac{1}{t_3}a_{12},0,0,\frac{1}{t_4}a_{22})$, what implies that the stabilizers of the orbits which maps to the boundary intervals of $P$ which contain the vertex $\delta _{34}$ are $(0,1)$ and $(1,0)$. By~\cite{DJ} we conclude   that our toric manifold is $\C P^1\times \C P^1$.  
\end{proof}
\begin{cor}
The closure of any two orbits from the stratum over $\Delta _{4,2}$ intersect along eight complex projective spaces $\C P^2$
which represent the closure of four-dimensional orbits.
\end{cor} 
\begin{rem}
Note that none of admissible polytope is tetrahedron implying that canonical action of $(\C ^{*})^{4}$ on $G_{4,2}$ has no orbits diffeomorphic to $\C P^3$.
\end{rem}
\section{The structure of the orbit space $G_{4,2}/T^4$}\label{orbit}

The results  obtained in Proposition~\ref{orbits} make us possible to describe the topological structure of the orbit space $G_{4,2}/T^4$. The main stratum is everywhere dense in $G_{4,2}/T^4$ and by Proposition~\ref{orbits} its orbits are parametrized by $c\in  \C -\{0,1\}$. Using this parametrization we will continuously parametrize by $c\in \C P^1$ all orbits for $(\C ^{*})^4$-action on $G_{4,2}$. For the simplicity we further use the notion $P_{ij}$ for the poytope spanned by the vertices  $\{\delta _{12},\ldots ,\delta _{34}\}-\delta _{ij}$ and by $P_{ij,pq}$ for the polytope spanned by the vertices  $\{\delta _{12},\ldots ,\delta _{34}\}-\{\delta _{ij},\delta _{pq}\}$, where $1\leq i<j\leq 4$, $1\leq p< q\leq 4$ and $(i,j)\neq (p,q)$.

\begin{prop}\label{orbits3}
The six-dimensional orbit for $(\C ^{*})^4$ - action on $G_{4,2}$, which do not belong to the main stratum, can be, depending on its  admissible polytope $P$, continuously parametrized  by $c=0,1,\infty$ using the parametrization of the orbits of the main stratum in the chart $M_{12}$ as follows:
\begin{itemize}
\item for $P = P_{14}$ or $P = P_{23}$ it is parametrized by $c=0$;
\item for  $P = P_{13}$ or $P = P_{24}$ it is parametrized by $c=\infty$;
\item for $P = P_{12}$ or $P = P_{34}$ it is parametrized by $c=1$.
\end{itemize}
\end{prop}
\begin{proof}
The orbits whose admissible polytopes are $P_{14}, P_{23}, P_{13}, P_{24}$ and $P_{34}$ belongs to the chart $M_{12}$.
Let us consider the orbit whose admissible polytope is $P_{14}$. It is given in the chart $M_{12}$ by  $\C ^{*}_{I}$, $I=\{1,2,3\}$. Assume  we are given  a sequence of points from the main stratum which converges to the point from this orbit, meaning that we are given a sequence $(z_1^n,z_2^n,z_3^n,z_4^n)\in (\C ^{*})^4$ which converges to the point $(z_1,z_2,z_3,0)$. It implies that the sequence of the parameters $c_n = \frac{z_1^nz_4^n}{z_2^nz_3^n}$, which parametrizes the orbits of the main stratum  converges to $0$. Therefore the orbit whose admissible polytope is $P_{14}$ is continuously parametrized by $c=0$. In the same way we argue for the orbits whose admissible polytopes are $P_{23}, P_{13}$ and $P_{24}$. These orbits are in  the chart $M_{12}$ given by $\C ^{*}_{I}$, where $I=\{2,3,4\}$, $I=\{1,2,4\}$, $I=\{1,3,4\}$, what implies that they are parametrized by $c=0$ and $c=\infty$ respectively.

The orbit whose admissible polytope is $P_{34}$ is given in the chart $M_{12}$ as the hypersurface 
$\frac{z_1z_4}{z_2z_3} = 1$, what implies that if the sequence of points $(z_1^n,z_2^n,z_3^n,z_4^n)$ from the main stratum converges to the point from this hypersurface, then the sequence of parameters converges to $1$. 

The orbit whose admissible polytope is $P_{12}$ does not belong to the chart $M_{12}$. We instead consider the chart which contains this orbit. Without loss of generality it is enough to consider the chart $M_{13}$. The local coordinates $(w_1,w_2,w_3,w_4)$ in the chart  $M_{13}$ and $(z_1,z_2,z_3,z_4)$ in the chart $M_{12}$ are related by
\[
w_1 = -\frac{z_1}{z_3},\; w_{2} = z_2 - \frac{z_1}{z_3}z_4,\; w_{3} = \frac{1}{z_3},\; w_4 = \frac{z_4}{z_3} .
\]
It follows if we note that points from the main stratum have in the charts $M_{12}$ and $M_{13}$ the matrix representations
\[
\left(\begin{array}{cc}
1 & 0\\
0 & 1\\
z_1 & z_3\\
z_2 & z_4 
\end{array}\right) \; \text{and}\;  
\left(\begin{array}{cc}
1 & 0\\
w_1 & w_3\\
0 & 1\\
w_2 & w_4 
\end{array}\right) ,
\]
what implies that the matrix of coordinate change is $\left(\begin{array}{cc}
1 & 0\\
w_1 & w_3 \end{array}\right)$.
 
We obtain that the parameters $c$ and $d$ of the orbits from the main stratum in the charts $M_{12}$ and $M_{13}$  are related by
\begin{equation}\label{param}
d = \frac{w_1w_4}{w_2w_3} = -\frac{z_1z_4}{z_2z_3-z_1z_4} = \frac{c}{c-1}.
\end{equation}
The orbit whose admissible polytope is $P_{12}$ is given in the chart $M_{13}$ by $\C ^{*}_{{\bar I}}$, ${\bar I}=\{1,2,4\}$. Therefore if we have a sequence of points from the main stratum which converges to the point from this orbit, it gives in the chart $M_{13}$ the sequence $(w_1^n, w_2^n, w_3^n, w_4^n)\in (\C ^{*})^4$ which converges to the point  $(w_1,w_2,0,w_4)$. It implies that the sequence of parameters $d_n$, which in the chart $M_{13}$ parametrizes the orbits of the main stratum, converges to $0$. If now consider the same sequence in the chart $M_{12}$,  using~\eqref{param}, we conclude that its corresponding sequence of parameters $c_n$  converges to $1$.
\end{proof}

\begin{prop}\label{orbitsl}
The $l$-dimensional orbit, where $l\leq 2$,  for $(\C ^{*})^4$ - action on $G_{4,2}$ can be, depending on $l$ and its admissible polytope $P$, continuously parametrized using the parametrization of the orbits of the main stratum in the chart $M_{12}$ as follows :
\begin{itemize}
\item if $l=0,1$ it is parametrized by any $c\in \C P^{1}$;
\item if $l=2$ and $P$ is a triangle it is parametrized by any $c\in \C P^1$;
\item if $l=2$ and $P$ is a square and  
\begin{enumerate}
\item $P = P_{14,23}$ it is parametrized by $c=0$,
\item $P = P_{13,24}$  it is parametrized by $c=\infty$,
\item $P = P_{12,34}$ it is parametrized by $c=1$.
\end{enumerate}
\end{itemize}
\end{prop}
\begin{proof}
The admissible polytope of $l$-dimensional orbit where $l=0,1$ is a vertex or edge of the boundary of $\Delta _{4,2}$,
what means that each such orbit is in the boundary of any orbit from the main stratum. Therefore if a point belong to the orbit whose  admissible  polytope is an edge which  
belongs to the chart $M_{12}$ or it is the vertex $\delta_{12}$, there is, for  any $c\in \C -\{0,1\}$, a sequence of points from the orbit of the main stratum parametrized by $c$ which converges to that point. Moreover, the vertex $\delta_{12}$ and the admissible edges with the vertex $\delta_{12}$ are in the boundary of five admissible three-dimensional polytopes with the vertex $\delta_{12}$. By Proposition~\ref{orbits3} it implies that there is  a sequence  of points from any of  orbits parametrized  by $c =0,1,\infty$ which converges to that point. Therefore, each point of such orbit can be continuously parametrized by any $c\in \C P^1$.
If an admissible  polytope for $l$-dimensional orbit where $l=0,1$  does not belong to the chart $M_{12}$, we instead consider the chart it belongs  and do the local coordinate change.

If an admissible polytope of two-dimensional orbit is a triangle, then such orbit is in the boundary of any orbit from the main stratum. Thus, if this triangle contains the vertex $\delta_{12}$ we can continuously parametrize this orbit, using the chart $M_{12}$, by any  $c\in \C -\{ 0,1\}$. Moreover, such triangle is in the boundary of three admissible three-dimensional polytopes and they are by Proposition~\ref{orbits3} all differently  parametrized. It implies that each point from this orbit can be parametrized by $c=0,1,\infty$.

If an admissible polytope of two-dimensional orbit is a square, then this orbit does not bound any orbit from the main stratum. It bounds exactly two three-dimensional orbits which, by Proposition~\ref{orbits3}, are parametrized by the same $c=0,1,\infty$. Let us consider such orbit whose admissible polytope is $P_{14,23}$. It belongs to the chart $M_{12}$ and it is represented by the points $(0,z_2,z_3,0)$, where $z_2,z_3\in \C ^{*}$. If $(z_1^n,z_2^n,z_3^n,z_4^n)$ is a sequence of points from the main stratum which converges to a point from this orbit we have  $z_1^n, z_4^n\rightarrow 0$, what implies that $c_n=\frac{z_1^nz_4^n}{z_2^nz_3^n}\rightarrow 0$. Thus, this orbit is continuously parametrized by $c=0$. In the same way we prove the case when the admissible polytope is $P_{13,24}$. When the admissible polytope is $P_{12,34}$, the orbit does not belong to the chart $M_{12}$. We similarly consider an arbitrary chart to which it belongs and do the local coordinate change. 
\end{proof}

Proposition~\ref{orbits}, together with the Proposition~\ref{orbits3} and Proposition~\ref{orbitsl} gives the topological description of the orbit space for compact torus action $T^4$ on $G_{4,2}$.

\begin{thm}\label{orbit-space-4}
The orbits space $X=G_{4,2}/T^4$ for the canonical action of the compact torus $T^4$ on $G_{4,2}$ is homeomorphic to the quotient space 
\begin{equation}
(\Delta _{4,2}\times \C P^1)/\approx \; \text{where}\; (x, c)\approx (y,c^{'})\Leftrightarrow x=y \in \partial \Delta _{4,2}.
\end{equation}
\end{thm}

Recall that the pair $(X, A)$ of topological spaces is called  a Lefschetz pair if $A$ is a compact subset in $X$ and $X-A$ is an open (non compact) manifold.  For $X$ being the orbit space we consider, Theorem~\ref{orbit-space-4} gives that $X- \partial \Delta _{4,2}$ is an open, non-compact manifold  $\stackrel{\circ}{\Delta _{4,2}}\times \C P^1$. As $\partial \Delta _{4,2}$ is topologically sphere $S^2$ it implies:

\begin{cor}
For the orbit space $X=G_{4,2}/T^4$ the pair $(X,  S^2)$ is a Lefschetz pair.
\end{cor}

We see that  the space $X$ is obtained by gluing all $\C P^1$'s copies of  $\Delta _{4,2}$ along their boundary $\partial \Delta _{4,2}$. Therefore $X$ consist of $\C P^1$'s copies of $\stackrel{\circ}{\Delta _{4,2}}$ and one copy of the space $\partial \Delta _{4,2}$. It implies obvious cell decomposition for $X$. 

 \begin{cor}\label{cell}
The orbit space $X=G_{4,2}/T^4$  has the cell decomposition with one cell in each of dimensions $5$, $3$, $2$ and $0$.
\end{cor}

This cell decomposition leads to the description of the homology groups for $X$.

\begin{prop}\label{orbit-H}
The orbit space $X=G_{4,2}/T^4$ is simply connected and, thus, orientable. Moreover $H_{0}(X)\cong H_{5}(X)\cong \Z$, while $H_{k}(X)\cong 0$ for $k\neq 0,5$. 
\end{prop}
\begin{proof}
The cell decomposition for $X$ given by Corollary~\ref{cell} implies that $\pi _{1}(X) = \pi _{1}(S^2)=0$. The homology groups for $X$ can be computed in the standard way. The groups  $C_{q}(X)=H_{q}(sk_{q}X, sk_{q-1}X)$ of the cell chain for $X$ are isomorphic to $\Z $ for  $q = 0,2,3,5$ while  otherwise they are trivial. It directly implies that $H_{1}(X)\cong H_{4}(X)\cong 0$ and $H_{0}(X)\cong H_{5}(X)\cong \Z$. The differential $\partial _{3} : C_{3}(X)\to C_{2}(X)$ is defined by the exact homology sequence of the triple $(sk_{3}X, sk_{2}X, sk_{1}X)=({\bar D}^{3}, S^2, \ast)$ which is $\ldots \rightarrow 0\cong H_{3}({\bar D}^3, \ast )\rightarrow H_{3}({\bar D}^{3}, S^2)\rightarrow H_{2}({\bar D}^3, \ast)\cong 0\rightarrow \ldots$. It implies that $\partial _{3} : C_{3}(X)\to C_{2}(X)$ is an isomorphism and therefore $H_{2}(X)\cong H_{3}(X)\cong 0$.  
\end{proof}

In this way we prove that the  orbit space $X$ is the homology sphere $S^5$.  The Hurewitz map together with homological Whitehead theorem implies:

\begin{cor}\label{homot-eq}
The orbit space $X=G_{4,2}/T^4$ is homotopy equivalent to $S^5$.
\end{cor}
\begin{proof}
We provide the standard proof for the sake of clearness. Since $X$ is a homology sphere $S^5$, Hurewicz theorem gives that $\pi _{k}(X)\cong 0$ for $k\leq 4$ and $\pi _{5}(X)\cong \Z$. Let $\alpha : S^5\to X$ be the generator for $\pi _{5}(X)$. Then $\alpha _{*} : \pi _{5}(S^5)\to \pi _{5}(X)$ is an isomorphism, since $\alpha _{*}([id])=[\alpha]$. Now the Hurewicz map induces the isomorphisms $h_{1} : \pi _{5}(S^5)\to H_{5}(S^5)$ and $h_{2} : \pi _{5}(X)\to H_{5}(X)$. It implies that the induced map $\alpha _{\bullet} : H_{5}(S^5)\to H_{5}(X)$ is an isomorphism. As $X$ is a homology sphere $\alpha_{\bullet}$ will be an isomorphism for all other homology groups as well, and by Whitehead homology theorem $X$ is homotopy equivalent to $S^5$.
\end{proof}

The previous Corollary can be proved in the following way as well.
\begin{cor}\label{join}
The orbit space $G_{4,2}/T^4$ is homeomorphic to the join $S^2\ast S^2$ and, thus, homotopy equivalent to $S^5$.
\end{cor}
\begin{proof}
We apply the general construction of homotopy theory. Namely,  given topological spaces $X$ and $Y$, one can consider the space $CX\times Y$, where $CX$ is a cone over $X$, where $(x_1,1)=(x_2,1)$ for $x_1,x_2\in X$. The quotient of the space $CX\times Y$  by the relation $(x_1,0,y_1)\approx (x_2,0,y_2)\Leftrightarrow  x_1=x_2$ is the join $X\ast Y$.  In our case Theorem~\ref{orbit-space-4} gives that $G_{4,2}/T^4$ is homeomorphic to the quotient of ${\bar D}^{3}\times S^2$ by the relation $(x_1,y_1)\approx (x_2,y_2) \Leftrightarrow x_1=x_2\in S^2 = \partial {\bar D}^{3}$. We identify ${\bar D}^{3}$ with $CS^2$ by $(1-t)x=xt$,  where $x\in S^2$ and $t\in [0, 1]$.  Therefore $G_{4,2}/T^4$  can be considered as the quotient of the space $CS^2\times S^2$ by the relation $(x_1,0,y_1)\approx (x_1,0,y_2)$.
It is well known that  $S^2\ast S^2$ is   homotopy equivalent to $S^5$.   
\end{proof}

We  prove further that the orbit space $X = G_{4,2}/T^4$ has a manifold structure.  The  gluing, given by Theorem~\ref{orbit-space-4},   of  the continuous family, parametrized by $\C P^1$, of  topological manifolds $\Delta _{4,2}$ along the boundary $\partial \Delta _{4,2}$ will produce a manifold without boundary. 
\begin{prop}\label{orbit-main}
The orbit space $X=G_{4,2}/T^4$ is a five dimensional manifold (topological, without boundary).
\end{prop}
\begin{proof}
 We need to find for any point $q\in X$ a  neighbourhood which is homeomorphic to an open five-dimensional disk. If $q\in \stackrel{\circ}{\Delta _{4,2}}\times \C P^1$,  it is obviously possible to do that as $\stackrel{\circ}{\Delta _{4,2}}$ is an open three-dimensional manifold, while $\C P^2$ is  a closed two-dimensional manifold. So, fix $q=[(x,c)]$, where $x\in \partial [\Delta _{4,2}]$.  Denote by $p : \Delta _{4,2}\times \C P^1 \to X$ the canonical projection map given by gluing from Theorem~\ref{orbit-space-4}.  Now if $U$ is an open neighbourhood for $q$ in $X$, then $p^{-1}(U)$ is an open subset in $\Delta _{4,2}\times \C P^1$ which contains $p^{-1}(q)=x\times \C P^1$. It gives that  $p^{-1}(U)=U_1\times \C P^1$, where $U_1$ is an open subset in $\Delta _{4,2}$ which contains $x$, so we may assume that $U_{1} = D^{3}_{\geq 0}$, a half-disk of an open disk $D^3$. Therefore $U$ is by $p$ homeomorphic to the quotient space of $D^{3}_{\geq 0}\times S^2$  obtained by the gluing of the points $(x,s)$ along $S^2$ for any fixed $x\in D^2$, where $D^2$ is the base of $D^3_{\geq 0}$:
\[
U\cong D^{3}_{\geq 0}\times S^2/\approx \; \text{where}\; (x,s)\approx (y,s)\Leftrightarrow x=y\in D^2.
\]      
The later quotient space is homeomorphic to an open disk $D^5$. 

We can see that homeomorphism as follows. It is obvious that $[0,1)\times S^2$ qoutiened by the relation $(0,s_1)\approx (0,s_2)$ for any $s_1,s_2\in S^2$  is homeomorphic to an open three-dimensional disk $D^3$. It implies that $D^{2}_{\geq 0}\times S^2$ qoutiened by the relation $(x_1,0,s_1)\approx (x_1,0,s_2)$ is homeomorphic to $D^4$. Namely, $D^2_{\geq 0}\times S^2$ is homeomorphic to $(0,1)\times [0,1)\times S^2$ and the relation translates as $(t_1,0,s_1)\approx (t_1,0,s_2)$. Therefore its quotient is homeomorphic to the space $(0,1)\times ([0,1)\times S^2)/\approx$, with the relation $(0,s_1)\approx (0,s_2)$. The later space is homeomorphic to $D^4$. Now, considering $D^{3}_{\geq 0}\times S^2$ as $(0,1)\times (0,1)\times [0,1)\times S^2$ and translating the relation we obtain this space to be homeomorphic to $(0,1)\times (0,1)\times ([0,1)\times S^2)/\approx$, what is $D^5$.   
\end{proof}

All together leads to complete topological characterization of $X$.

\begin{thm}\label{top-main}
The orbit space $X=G_{4,2}/T^4$ is homeomorphic to the sphere $S^5$.
\end{thm}
\begin{proof}
By Proposition~\ref{orbit-H}, Proposition~\ref{orbit-main} and Corollary~\ref{homot-eq}, the orbit space $X$ is simply-connected closed manifold homotopy equivalent to the sphere $S^5$. Then the generalized Poincare conjecture~\cite{Smale},~\cite{Stallings} implies that $X$ is homeomorphic to $S^5$.
\end{proof} 

\begin{rem}
We want to point that the decomposition of $X$ induced by the orbits of $(\C ^{*})^4$-action on $G_{4,2}$ or equivalently by the admissible polytopes in $\Delta _{4,2}$ will not be a cell decomposition of $X$. The moment map $\mu : G_{4,2}\to \Delta _{4,2}$ is $T^4$-invariant and using canonical projection $p : G_{4,2}\to X=G_{4,2}/T^4$, it  induces a continuous map $q : X \to \Delta _{4,2}$ by $q(x)=\mu (p^{-1}(x))$. In this way the map $q$ gives that $X$ can be decomposed as the union of the infinitely many sets homeomorphic  to the interior of octahedron, $6$ sets homeomorphic to the interior of the admissible pyramids, $8$ sets  homeomorphic to the interior of the admissible triangles, $3$ sets  homeomorphic to the interior of the admissible squares, $12$ sets homeomorphic to the interior of the admissible intervals and $6$ points  that correspond to the fixed points. These sets are glued to each other along the boundary of the corresponding polytopes using the map $q$. We obtain the decomposition of $X$ induced by the orbits of $(\C ^{*})^4$-action on $G_{4,2}$ or equivalently by the admissible polytopes in $\Delta _{4,2}$ and it is not a cell decomposition of $X$. More precisely, the cell axiom will be  satisfied but  the  weak topology is not satisfied. Namely, $X$ is compact in the quotient topology $X$ inherits from $G_{4,2}$, while in the case this is a cell decomposition of $X$, we would have that $X$ is non-compact neither locally compact. 
\end{rem}

\begin{rem}
Note that in the case of a toric manifold $X$, on the contrary,  the decomposition of the  corresponding polytope into admissible polytopes, defined by the moment map for $X$,
  gives the cell decomposition of the quotient $X/T$, where $T$ is a compact form of an algebraic torus  acting on $X$.
\end{rem}

\subsection{On the existence of  a section} It follows from Theorem that the map $\mu : G_{4,2}/T^4 \to \Delta _{4,2}$ is not one to one, so it is not possible to construct a section $s: \Delta _{4,2}\to G_{4,2}$ as it is done in the case of toric manifolds. Therefore our idea is to unpack the octahedron $\Delta _{4,2}$ into cell complex of admissible polytopes. 

Denote by $\mathfrak{S}$ the  family of admissible polytopes for $G_{4,2}$. Note that it consists of six points, twelve edges, eight triangles, three squares, six pyramids and one octahedron.

It is defined the operator 
\[
d : \mathfrak{S}\to S(\mathfrak{S})\;
\text{by}\; dP\; \text{is disjoint union of the faces of}\; P.
\]

In this way we obtain cell complex $\mathfrak{W}$ whose  cells are the admissible polytopes from $\mathfrak{S}$ and we glue them by induction using the operator $d$.  

By the construction we see that there is the canonical map $\pi : \mathfrak{W}\to \Delta _{4,2}$ and that the moment map $\mu : G_{4,2}\to \Delta _{4,2}$ extends to the map $\mu _{\mathfrak{W}} : G_{4,2}\to \mathfrak{W}$ such that 
$\mu = \pi \circ \mu _{\mathfrak{W}}$. 

Consider the set $X = G_{4,2} - W_{\Delta _{4,2}}$, where as before $W_{\Delta _{4,2}}$ denotes the main stratum. The set $X$ is a closed subset in $G_{4,2}$ and it is invariant under the action of $T^4$ and $\mu _{\mathfrak{W}}(X) = \mathfrak{Q} = \mathfrak{W} - \stackrel{\circ}{\Delta}_{4,2}$, which is again a cell complex. Since any two point from $G_{4,2}$ which have the same admissible polytope have  the same stationary subgroup in $T^4$, it is defined the characteristic function $\chi : \mathfrak{Q}\to S(T^4)$ by 
$\chi (q) = \chi (x)$ for $x\in G_{4,2}$ such that $\mu _{\mathfrak{W}}(x) =q$. 

It further follows from Corollary~\ref{Gr_other_strata} that the map $\mu _{\mathfrak{W}}(X/T^4)\to \mathfrak{Q}$ is one to one, what implies that there exists the section $s: \mathfrak{Q}\to X$ given by $s(q) = (q,1)$. As in the case of toric manifolds~\cite{BP} it  implies that $X$ is equivariantly homeomorphic to
\begin{equation}\label{novi}
X = T^4 \times \mathfrak{Q}/\approx,\;\; (t_1,q_1)\approx (t_2,q_2) \Leftrightarrow q_1=q_2,\;\text{and}\; t_1t_2^{-1}\in \chi (q_1).
\end{equation}
Therefore $G_{4,2}$ is obtained by gluing to the main stratum   $W_{\Delta _{4,2}}$, which is homeomorphic to $T^3 \times \stackrel{\circ}{\Delta}_{4,2}\times (\C -\{0,1\})$, the set given by~\eqref{novi}.

\section{The smooth and singular points of the orbit space $G_{4,2}/T^4$}

\subsection{Some general facts.}\label{gen-smth} In order to determine the smooth and singular points on $X$, related to the smooth structure on $G_{4,2}$, we make an explicit  use of the slice or equivariant tubular neighbourhood theorem due to~\cite{Kosz} and~\cite{Pal}, see also~\cite{BR}. From that reason we recall it in more detail. 

Assume we are given a smooth manifold $M$ with a smooth action of a compact group $G$ and let $p\in M$. 
If $p$ is a fixed point for the given action then the slice theorem or so called local linearization theorem states:
\begin{thm}
For a fixed point $p$ there exist $G$-equivariant diffeomorphism from a neighbourhood of the origin in $T_{p}M$ onto neioghbourhood of $p$ in $M$.
\end{thm}

If $p$ is not a fixed point, denote by $H$ its stabilizer, which is a proper subgroup of $G$. Let further $TM/(G\cdot p)$ denotes the tangent bundle for $M$ along the points of orbit $G\cdot p$.   The slice  representation $V$ for $p$ is defined to be  the normal bundle in $TM/(G\cdot p)$  to the tangent bundle $T(G\cdot p)$ of the orbit. It is taken related to some $G$-invariant metric on $M$. Then we  have the representation of $H$ in $V$.  

The general slice theorem states:
\begin{thm}
There exists $G$-equivariant diffeomorphism from the vector bundle $G\times _{H}V$ onto neighbourhood of the orbit $G\cdot p$ in $M$.
\end{thm}

The  diffeomorphism $\psi : G\times _{H}V\to U_{M}(G\cdot p)$ in the slice theorem is given by 
$\psi ([(g,v)])=g\cdot \varphi (v)$, where $\varphi : U_{T_{p}M}(0)\to U_{M}(p)$ is an $H$-equivariant diffeomorphism from the local linearization theorem. The diffeomorphism $\psi$ is $G$-equivariant since an action of $G$  on $G\times _{H}V$ is given by $g_1\cdot [(g,v)]=[(g_1g,v)].$

The action of $H$ on $V$ induces the  decomposition $V=V^{H}\oplus L$ related to some $G$-invariant metric, where $V^{H}$ is the subspace of the vectors fixed by $H$.  Therefore it holds
\[
G\times _{H}V=G\times _{H}(V^{H}\oplus L) = V^{H}\times (G\times _{H}L).
\]
It can be easily seen that $[(g,v)]=[(hg,v)]$ in $G\times _{H}V$ for any $g\in G$, $h\in H$ and for any $v\in V^{H}$. Thus, $[(g,v_1,v_2)]=[(g_1,v_1^{'},v_2^{'})]$ in $G\times _{H}(V^{H}\oplus L)$ if and only if $v_1=v_1^{'}$ and $[(g,v_2)]=[(g_1,v_2^{'})]$ in $G\times _{H}L$. 

If further implies that
\[
(G\times _{H}V)/G = V^{H}\times L/H.
\]
Namely the action of $G$ on $G\times _{H}V=V^{H}\times (G\times _{H}L)$ is given by $g\cdot (v_1, [(g_1,v_2)])=(v_1,[(gg_1,v_2)])$ what implies that $(G\times _{H}V)/G = V^{H}\times (G\times _{H}L)/G$. Now the orbit of an element $[(g_{0},v_{0})]\in G\times _{H}L$ by the action of the group $G$ is $[(g,v_0)]$, where $g\in G$. It implies that the elements $[(g_1,v_1)]$ and $[(g_2,v_2)]$ are in the same $G$-orbit in $G\times _{H}L$ if and only if $v_1$ and $v_2$ are in the same orbit for the $H$-action on $L$.

Recall that we have fixed some $G$-invariant Riemannian metric what gives that the action of $H$ on $L$ preserves the scalar product meaning that $H(S(L))\subseteq S(L)$ where $S(L)$ is an unitary  sphere whose center is at origin $p$ of $L$. Therefore 
\[
L/H = ([0,\infty)\times S(L))/H = \cone(S(L)/H),
\]
what gives
\begin{equation}\label{ndif}
(G\times _{H}V)/G = V^{H}\times \cone (S(L)/H).
\end{equation} 

\subsection{Application to the orbit space $G_{4,2}/T^4$.} 
By Theorem~\ref{top-main} we have that $X=G_{4,2}/T^4$ is homeomorphic to the sphere $S^5$. On the other hand it is the classical result~\cite{KER-MIL},~\cite{LEV} that the sphere $S^5$ has unique differentiable structure, which is the standard one. It suggests that the differentiable structure  on $G_{4,2}$ does not induce the differentiable structure on $X$ meaning that there is no smooth structure on $X$ such that the natural projection $\pi : G_{4,2}\to X$ is a smooth map. Otherwise $X$  would be  diffeomorphic to the standard sphere $S^5$ for which it is known to admits the smooth action of the circle $S^1$, while it is not clear where such an action on $X$ would come from.  

\begin{rem}
We can interpret this fact using the notion of functional "smooth structure" as defined in~\cite{BR}. The functional structure $\mathfrak{F}(G_{4,2})$  which is induced from the smooth structure of $G_{4,2}$ gives, by the push-forward, the "smooth structure"  $\pi _{*}\mathfrak{F}(G_{4,2})$ on $X$ and $\pi : G_{4,2}\to X$ is a "smooth map" related to the functional structures on these spaces. On the other hand $X\cong S^5$ has the functional  structure $\mathfrak{F}(S^5)$ induced from the unique smooth structure on $S^5$. In this terminology we actually prove that the natural projection $\pi : (G_{4,2}, \mathfrak{F}(G_{4,2}))\to (S^5, \mathfrak{F}(S^5))$ is not a morphism of the functional structures. It, in particulary, implies that $\pi : G_{4,2}\to X$ is not a smooth map. 
\end{rem}

We prove  that the quotient structure on $G_{4,2}/T^4$ is not differentiable and describe the corresponding smooth and singular points.  We do it by applying  the given general diffeomorphic description of the equivariant neighbourhood of the orbits for the smooth action of the compact group.

Denote by  $\pi : G_{4,2}\to G_{4,2}/T^4$ the  natural projection.
\begin{thm}\label{smain}
The point $q\in G_{4,2}/T^4$ is:
\begin{itemize}
\item a smooth point if $\dim \pi ^{-1}(q) = 3$;
\item a cone-like singularity point if $\dim \pi ^{-1}(q)\leq 2$ and has  neighbourhood of the form 
\begin{enumerate}
\item $D^2\times \cone (S^2)\;\;$ for $\;\;\dim \pi ^{-1}(q)=2$;
\item  $D^1\times \cone (S^5/T^2)\;\;$ for $\;\;\dim \pi ^{-1}(q) =1$;
\item $\cone (S^7/T^3)\;\;$ for $\;\; \dim \pi ^{-1}(q)=0$,
\end{enumerate}
where the induced actions on $S^5$ and $S^7$ are without fixed points.
\end{itemize}
\end{thm} 

\begin{proof}
Let $p\in G_{4,2}$  such that $\pi (p) =q$ and let $H\subseteq T^3$ be the stabilizer of the point $p$ for an effective action of the torus $T^3$ on $G_{4,2}$. The following cases are possible:

{\bf 1)} $H=\{ e\}$  is trivial. The slice theorem implies that there exists $T^3$-equivariant neighbourhood $U(p)$ of the orbit $T^3\cdot p$ in $G_{4,2}$  which is $T^3$-equivariantly diffeomorphic to $T^3\times _{e}V = T^3\times V$, where
$V=T_{p}G_{4,2}/T_{p}(T^3\cdot p)$ related to some $T^3$-invariant metric. It implies that $U(p)/T^3$ has the same singular and smooth points as  $(T^3\times V)/T^3 = V$. Therefore the point $q$ is a smooth point.

{\bf 2)} $H=T^1$ is one-dimensional torus. The orbit $T^3\cdot p$ is two-dimensional what implies that $V=T_{p}G_{4,2}/T_{p}(T^3\cdot p)$ is six-dimensional. It follows from~\eqref{ndif} that there exists a neighbourhood $U(p)$ of the orbit $T^3\cdot p$ such that $U(p)/T^3$ has the same singular and smooth points as   $V^{T^1}\times \cone (S(L)/T^1)$.
In order to determine the dimension of $V^{T^1}$ note  that any point of the  algebraic torus orbit $(\C ^{*})^4\cdot p$, which is four-dimensional,  is invariant under the action of a  compact stabilizer $T^1$ for $p$. It gives that $\dim V^{T^1}\geq 2$. Also  if $T^1$ fixes a point $s\in G_{4,2}$ then $\dim (\C ^{*})^4\cdot s \leq 2$.   On the other hand there is a neighbourhood of an orbit $(\C ^{*})^4\cdot s$  which does not intersect any of the orbits of   smaller dimensions. It can be seen for example by looking at the  local chart which contains this orbit. Thus, $\dim V^{T^1}=2$. It implies that $\dim L=4$ and 
$V^{T^1}\times \cone (S(L)/T^1)$ has the same smooth and singular points as 
\[
D^{2}\times \cone (S^3/T^1) = D^2\times \cone (S^2),
\]
since $T^1$ acts here on $S^3$ without fixed points. Therefore, the orbit space $X=G_{4,2}/T^4$ has  cone-like singularities at the points whose stabilizer is $T^1$.

{\bf 3)} $H=T^2$ is two-dimensional torus. The orbit $T^3\cdot p$ is one-dimensional and $V$ is then seven-dimensional. The neighbourhood of $U(p)/T^3$ has the same smooth and singular points as   $V^{T^2}\times \cone (S(L)/T^2)$. As in the previous case we deduce that $\dim V^{T^2} = 1$ implying that $\dim L=6$. Therefore we obtain that $U(p)/T^3$ has the same smooth and singular points as 
\[
D^1\times \cone (S^5/T^2).
\]

{\bf 4)} Let $H = T^3$. Then $p$ is a fixed point and, as the local linearization theorem says, $U(p)/T^3$ has the same smooth and singular points as   $(T_{p}G_{4,2})/T^3 = \cone (S^7/T^3)$, where torus $T^3$ acts on $S^7$ without fixed points. 
\end{proof}
We describe further the differentiable structure of the orbit spaces $S^5/T^2$ and $S^7/T^3$ from the previous theorem.
\begin{prop}\label{first}
The orbit space $S^5/T^2$ has three cone-like singular points, while all other points are smooth. Moreover, the singular points have a neighbourhood of the form $\cone (S^2)$.
\end{prop}
\begin{proof}
We use the notation from the proof of Theorem~\ref{smain}. Thus, we  consider the point $p\in G_{4,2}$ such that $\dim ((\C ^{*})^4\cdot p) =1$.   Consider the chart for $G_{4,2}$ where the point $p$  has the coordinates $(z_1,0,0,0)$.  Then the sphere $S^5\subseteq V$ is the sphere in the  subspace  $\C ^{3}\subseteq \C ^{4}$ which is given by the points $(0,z_2,z_3,z_4)$. The action of $T^2$ on this $\C ^{3}$ is given by $(t_2,t_3)\cdot (z_2,z_3,z_4)=(t_2z_2,t_3z_3,t_2t_3z_4)$. Thus the points $(z_2,0,0)$, $(0,z_3,0)$ and $(0,0,z_4)$ from $\C ^3$ have one-dimensional orbits, while the orbits for all other points  are two-dimensional. Now it is clear that the points on $S^5/T^2$ which are obtained as the  quotient of two-dimensional orbits are smooth. Therefore let $s\in S^5$ be a point whose $T^2$-orbit is one-dimensional. Then the dimension of the slice representation $V$ is four and the orbit space of the  neighbourhood of $T^2\cdot s$ by the action of $T^2$ has the same smooth and singular points as
 \[
 V^{T^1}\times \cone (S^3/T^1) = \cone (S^2),
 \]
 since the only points fixed by $T^1$ are from the orbit $T^3\cdot s$. It implies that  $S^5/T^2$ has three cone-like singular points which correspond to the one-dimensional  orbits on $S^5$.    
\end{proof}        

\begin{prop}
The points of the orbit space $S^7/T^3$  which correspond to the:
\begin{itemize}
\item  three-dimensional orbits are smooth points;
\item two-dimensional orbits are cone - like singularities with a neighbourhood of the form $D^1\times \cone (S^2)$;
\item one-dimensional orbits are cone-like singularities  with a neighbourhood of the form $\cone (S^5/T^2)$.
\end{itemize}
\end{prop}

\begin{proof} 
Consider as earlier the chart on $G_{4,2}$ which contains $p$. In this chart $p=(0,0,0,0)$. Related to the induced action of $T^3$, the points from $S^7\subseteq \C ^4$ which are on the  coordinate axis have one-dimensional orbits, those which are on the two-dimensional coordinate planes have two-dimensional orbits, while all  other points have three-dimensional orbits.   
Thus, there are four one-dimensional orbits and six two-dimensional orbits.  For a point $s$ whose orbit is  one-dimensional  the orbits space $U(s)/T^3$ has the same smooth and singular points as $\cone (S^5/T^2)$. By Proposition~\ref{first}  we deduce that $U(s)/T^3$ has cone-like singularity at the points $s$ and $(0,\infty)\times S_i$, where $S_i$, $1\leq i\leq 3$ are the singular points on $S^5/T^2$.

If the orbit of a point $s$ is two-dimensional, then the slice representation $V$ for $s$ is five-dimensional.
Moreover $V^{T^1}\subseteq (\C ^2\cap S^7)/T^3$ is one-dimensional what implies that $L$ is four-dimensional.
Therefore the  orbit space $U(s)/T^3$ has the same smooth and singular points  as $D^1\times \cone (S^3/T^1)= D^1\times \cone (S^2)$.

\end{proof}

\begin{rem}
By Theorem~\ref{orbit-space-4} we have  that $X=G_{4,2}/T^4$ is homeomorphic to   $\Delta _{4,2}\times \C P^1$ where all copies of the hypersimplices $\Delta _{4,2}$ are glued along their common boundary.  Now $\Delta _{4,2}\times \C P^1$ is a manifold  with the singular points of the form $(v,c)$, for any vertex $v$  of $\Delta _{4,2}$ and any $c\in \C P^1$. Namely, at each vertex $v$ there are four edges $e_1,\ldots ,e_4$ which on $X$ gives four "cylinders" $e_i\times \C P^1$, $1\leq i\leq 4$. Each of  three of these cylinders  intersect on $\Delta _{4,2}\times \C P^1$ transversally at $v$. Since there are four of them and, in the case of smoothness,  the dimension of the tangent space at any point of $\Delta _{4,2}\times \C P^1$ has to be $\leq 5$,  we see that at the points $v\times \C P^1$ we have cone-like singularity on $\Delta _{4,2}\times \C P^1$.           
Note that this argument will not work any more if we consider the points $v\times \C P^1$ on $X=G_{4,2}/T^4$, for a vertex $v$ of $\Delta _{5,2}$. Namely, at these points each of the corresponding edges  $e_i\times \C P^1=e_i$, $1\leq i\leq 4$, is one-dimensional, while $X$ is five-dimensional manifold.
\end{rem}

\section{The orbit space  $\C P^5/T^4$}\label{complex}
Consider the embedding of $(\C ^{*})^{4}$ into $(\C ^{*})^6$ given by the second symmetric power:  
\begin{equation}\label{emb}
(t_1,t_2,t_3,t_4)\to (t_1t_2,t_1t_3,t_1t_4,t_2t_3,t_2t_4,t_3t_4).
\end{equation}
The composition of this embedding with the standard action of the algebraic torus $(\C ^{*})^{6}$ on  $\C P^5$ gives the action of $(\C ^{*})^{4}$ on $\C P^5$.

We  describe the combinatorics of the stratification for this  $(\C ^{*})^{4}$ - action on $\C P^5$. 
The weights for the embedding~\eqref{emb} of $(\C ^{*})^4\subset (\C ^{*})^6$   are: 
\[
\alpha _{ij} = x_i+x_j,\;\; 1\leq i < j\leq 4.
\]
It implies that $(\C ^{*})^4$ is a regular subtorus of  $(\C ^{*})^6$, so $(\C ^{*})^4$-action and the standard $(\C ^{*})^6$-action on $\C P^5$ have the same set of fixed points.
  
By~\cite{Kir} the moment map $\mu : \C P^5 \to \R ^4$  for $(\C ^{*})^{4}$-action   is given by
\begin{equation}\label{CP-moment}
\mu ([(z_1,\ldots ,z_6)]) = \frac{|z_1|^2(1,1,0,0)+\ldots + |z_6|^2(0,0,1,1)}{\sum\limits_{i=1}^{6}|z_i|^2}.
\end{equation}
The image of the moment map is the octahedron $\Delta _{4,2}$. As for the admissible polytopes we prove the following.
\begin{lem}
Every  convex polytope spanned by some subset of vertices for $\Delta _{4,2}$ is admissible polytope for the moment map
of the $(\C ^{*})^4$-action on $\C P^5$.
\end{lem} 
\begin{proof}
Take a convex polytope over some subset of vertices $\delta_{i_1j_1},\ldots \delta_{i_s,j_s}$, $1\leq s\leq 6$ for $\Delta _{4,2}$. Consider the point $[(z_{11},z_{12},\ldots ,z_{34})]\in \C P^5$, where $z_{ij}=1$ if $\{i, j\} = \{ i_l,j_l\}$ for some $1\leq l\leq s$, while otherwise $z_{ij}=0$. It is straightforward to see that $(\C ^*)^{4}$ - orbit of this point maps by the moment map to the convex polytope spanned by $\delta_{i_l,j_l}$, $1\leq l\leq s$.
\end{proof}

Consider the chart $M_0$ on $\C P^{5}$ given by the condition $z_0\neq 0$ in homogeneous coordinates for $\C P^{5}$, meaning that  
\[
M_{0}= \{ [(1,z_1,,\ldots ,z_5)], z_1,\ldots ,z_5\in \C \}.
\] 
Let us fix a point $[(1,a_1,\ldots ,a_5)]\in M_0$. Its $(\C ^{*})^4$-orbit is given by
\[
[(t_1t_2,t_1t_3a_1,t_1t_4a_2,t_2t_3a_3,t_2t_4a_4,t_3t_4a_5)] = [(1,\frac{t_3}{t_2}a_1, \frac{t_4}{t_2}a_2, \frac{t_3}{t_1}a_3,  \frac{t_4}{t_1}a_4, \frac{t_3t_4}{t_1t_2}a_5)].
\]
If we put 
\[\tau _1 =  \frac{t_3}{t_2},\; \tau _2 =  \frac{t_4}{t_2},\; \tau _3 =  \frac{t_3}{t_1},
\]
this orbit writes as
\[
(\C ^{*})^4 [(1,a_1,\ldots,a_5)] = [(1, \tau _1a_1, \tau _2a_2,\tau _3a_3,\frac{\tau _{2}\tau _{3}}{\tau _1}a_4, \tau _{2}\tau _{3}a_5)].
\]
We first note  any polytope which correspond to such orbit contains the vertex $\delta_{12}$.

As in the case of $G_{4,2}$ we denote further by $P_{ij}$ and $P_{ij,kl}$, $i<j$, $k<l$  the polytopes spanned by the vertices of $\Delta _{4,2}$ different from the vertex $\delta_{ij}$ and $\delta_{ij}$, $\delta_{kl}$ respectively.      

\begin{lem}
If $a_1\cdots a_5\neq 0$ then the orbit  $(\C ^{*})^4\cdot [(1,a_1,\ldots ,a_5)]$ is given by the surface 
\begin{equation}\label{main}
\frac{z_2z_3}{z_5} = \frac{a_2a_3}{a_5} = c_1,\;\; \frac{z_1z_4}{z_5} = \frac{a_1a_4}{a_5} = c_2,
\end{equation}
where $c_1c_2\neq 0$. Moreover
\[
\mu ((\C ^{*})^4[(1,a_1,\ldots ,a_5)]) = \stackrel{\circ}{\Delta _{4,2}}.
\]
\end{lem}
Thus the orbits from the main stratum are parametrized by $(c_1,c_2)$, $c_1,c_2\in \C ^{*}$. The orbits are six-dimensional what implies that the main stratum is ten-dimensional. Because of the further purposes we use
the parametrization of the orbits from the main stratum in the form $(c_1,c_2, \frac{c_1}{c_2})$, $c_1,c_2\neq 0$.

\begin{lem}
The six-dimensional orbits  $\OO _{\C}({\bf a}) = (\C ^{*})^4[(1,a_1,\ldots ,a_5)]$ in the chart $M_{0}$  which do not belong to the main stratum are   as follows:
\begin{itemize}
\item for $a_1=0$ or $a_4=0$ it is 
\[
z_1=0 \;\text{or}\; z_4=0\;  \text{and}\; \frac{z_2z_3}{z_5} = \frac{a_2a_3}{a_5} = c_1,
\]
where $c_1\neq 0$. In this case, respectively
\[
\mu (\OO _{\C}({\bf a})) = \stackrel{\circ}{P_{13}}\;\; \text{or}\;\; \mu (\OO _{\C}({\bf a})) = \stackrel{\circ}{P_{24}};
\]
\item  for $a_2=0$ or $a_3=0$ it is 
\[
z_2=0 \; \text{or}\; z_3=0\; \text{and}\; \frac{z_1z_4}{z_5} = \frac{a_2a_3}{a_5} = c_2,
\]
where $c_2\neq 0$. In this case
\[
\mu (\OO _{\C}({\bf a})) = \stackrel{\circ}{P_{14}}\; \text{or}\; \mu (\OO _{\C}({\bf a})) = \stackrel{\circ}{P_{23}};
\]
\item for $a_5 = 0$ it is 
\[
z_5=0 \; \text{and}\; \frac{z_1z_4}{z_2z_3} = \frac{a_1a_4}{a_2a_3} = c,
\]
where $c\neq 0$.  
\[
\mu (\OO _{\C}({\bf a}) ) = \stackrel{\circ}{P_{34}},
\]
\item for $a_1=a_2=0$  it is  
\[
\OO _{\C}({\bf a})= \C ^{*}_{I},\; I=\{3,4,5\}\;\;
\text{and}\;\; 
\mu (\OO_{\C}({\bf a})) = \stackrel{\circ}{P_{13,14}}.
\]
\item  for  $a_1=a_3=0$ it is 
\[
\OO_{\C}({\bf a})= \C ^{*}_{I},\; I=\{2,4,5\}\;\;
\text{and}\;\;
 \mu (\OO_{\C}({\bf a})) = \stackrel{\circ}{P_{13,23}}.
\]
\item  for  $a_1=a_5=0$ it is 
\[
\OO_{\C}({\bf a}) = \C ^{*}_{I},\; I=\{2,3,4\}\;\;
\text{and}\;\; 
 \mu (\OO_{\C}({\bf a})) = \stackrel{\circ}{P_{13,34}}.
\]
\item  for  $a_2=a_4=0$ it is g
\[
\OO_{\C}({\bf a}) = \C ^{*}_{I},\; I=\{1,3,5\}\;\;
\text{and}\;\; 
 \mu ( \OO_{\C}({\bf a}))= \stackrel{\circ}{P_{14,24}}.
\]
\item for $a_2=a_5=0$ it is 
\[
\OO_{\C}({\bf a}) = \C ^{*}_{I},\; I=\{1,3,4\}\;\;
\text{and}\;\; 
 \mu (\OO_{\C}({\bf a})) = \stackrel{\circ}{P_{14,34}}.
\]
 \item  for $a_3=a_4=0$ it is 
\[
\OO_{\C}({\bf a}) = \C ^{*}_{I},\; I=\{1,2,5\}\;\;
\text{and}\;\; 
 \mu (\OO_{\C}({\bf a})) = \stackrel{\circ}{P_{23,24}},
\]
\item  for $a_3=a_5=0$ it is 
\[
\OO_{\C}({\bf a}) = \C ^{*}_{I},\; I=\{1,2,4\}\;\;
\text{and}\;\; 
 \mu (\OO_{\C}({\bf a})) = \stackrel{\circ}{P_{23,34}}.
\]

\item  for  $a_4=a_5=0$ it is 
\[
\OO_{\C}({\bf a}) = \C ^{*}_{I},\; I=\{1,2,3\}\;\;
\text{and}\;\; 
 \mu ( \OO_{\C}({\bf a}))= \stackrel{\circ}{P_{24,34}}.
\]

\end{itemize}
\end{lem}

Note that we obtain that the strata over the polytopes $P_{13},P_{24},P_{14},P_{23}$ and $P_{34}$ are one-parameter strata and, thus, eight-dimensional. The strata over the other polytopes from previous lemma consists of one $(\C ^{*})^4$-orbit and they are six-dimensional. 
Using~\eqref{main} we obtain:
\begin{cor}
The six-dimensional orbits for $(\C ^{*})^4$ - action on $\C P^5$ of the strata different form the main stratum and which belong to the chart $M_0$ can be, depending on its  admissible polytope $P$, continuously parametrized using the parametrization of the orbits from the main stratum as follows:
\begin{itemize}
\item $P_{14}, P_{23}$ $\longrightarrow$  $(0, c_2,0)$, $c_2\in \C ^{*}$;
\item $P_{13}, P_{24}$ $\longrightarrow$  $(c_1,0,\infty)$, $c_1\in \C ^{*}$;
\item $P_{34}$ $ \longrightarrow$ $(\infty ,\infty ,c)$, $c\in \C P^1$;
\item $P_{13,14}$, $P_{13,23}$, $P_{14,24}$, $P_{23,24}$ $\longrightarrow$ $(0,0,c)$, $c\in \C P^1$;
\item $P_{13,34}$, $P_{24,34}$ $\longrightarrow$ $(\infty ,c,\infty )$, $c\in \C P^1$;
\item $P_{14,34}$, $P_{23,34}$ $ \longrightarrow$ $(c,\infty ,0)$, $c\in \C P^1$.
\end{itemize}
\end{cor}
\begin{rem}
Note that the polytopes $P_{14},P_{23},P_{13},P_{24},P_{34}$ are four-sided pyramids, while the polytopes $P_{13,14}$, $P_{13,23}$, $P_{14,24}$, $P_{23,24}$, $P_{13,34}$, $P_{24,34}$
$P_{14,34}$, $P_{23,34}$ are tetrahedra in $\Delta _{4,2}$ which contain the vertex $\delta _{12}$.
\end{rem} 
\begin{lem}
The four-dimensional orbits  $\OO _{\C}({\bf a}) = (\C ^{*})^4[(1,a_1,\ldots ,a_5)]$ in the chart $M_{0}$   are   as follows:
\begin{itemize}
\item   for $a_1=a_4=0$ it  is given by
\[
\OO _{\C}({\bf a}) =\{(0,z_2,z_3,0, z_5) : \frac{z_2z_3}{z_5}=c_1\}\;\;
\text{and} \;\;
 \mu (\OO _{\C}({\bf a})) = \stackrel{\circ}{P_{13,24}};
\]
\item  for $a_2=a_3=0$ it is given by
\[
\OO _{\C}({\bf a}) = \{(z_1,0,0,z_4,z_5) :  \frac{z_1z_4}{z_5}=c_2\}\;\;
\text{and}\;\;
 \mu (\OO _{\C}({\bf a})) = \stackrel{\circ}{P_{14,23}}.
\]
\end{itemize}
\end{lem}
Note that the strata over the polytopes $P_{13,24}$ and $P_{14,23}$ are one-parameter strata and they are six-dimensional. 
\begin{cor}
The four-dimensional orbits for $(\C ^{*})^4$ - action on $\C P^5$, which belong to the chart $M_0$ can be, depending on its  admissible polytope $P$, continuously parametrized using the parametrization of the orbits from the main stratum as follows:
\begin{itemize}
\item $P_{13,24} \longrightarrow (c_1,0,\infty)$, $c_1\in \C ^{*}$;
\item $P_{14,23}\longrightarrow (0,c_2,0)$, $c_2\in \C ^{*}$.
\end{itemize}
\end{cor}
\begin{rem}
Note that the polytopes $P_{13,24}$ and $P_{14,23}$ are the diagonal rectangles in $\Delta _{4,2}$ which contain the vertex $\delta _{12}$.
\end{rem}

The two-dimensional orbits for $(\C ^{*})^4$ action on $\C P^5$ which belong to the chart $M_0$ are given by $\{(z_1,\ldots ,z_5),\; z_i=z_j=z_k=0,\; z_l\neq 0,\; l\neq i,j,k\}$, $1\leq i<j<k\leq 5$.   
It follows from~\eqref{CP-moment} that the admissible polytopes for such strata are all triangles in $\Delta _{4,2}$ which contain the vertex $\delta _{12}$. According to previous notation by  $P_{ij, kl, pq}$ we denote triangle which does not contains the vertices $\delta_{ij}, \delta_{kl}$ and $\delta_{pq}$. Any stratum over a triangle is one orbit stratum and it is four-dimensional.
\begin{lem}
The strata which contain  four-dimensional orbits for $(\C ^{*})^4$ - action on $\C P^5$, which belong to the chart $M_0$ can be, depending on its  admissible polytope $P$, continuously parametrized using the parametrization of the orbits from the main stratum as follows:
\begin{itemize}
\item $P_{13,14, 23}$ $\longrightarrow$ $(0,0,c), c\in \C P^1$,
\item $P_{13,14, 24}$ $\longrightarrow$ $(0,0,c), c\in \C P^1$,
\item $P_{13,14, 34}$ $\longrightarrow$ $(c_1,c_2,\frac{c_1}{c_2})$,  
\item $P_{13, 23, 24}$ $\longrightarrow$ $(0,0,c)$, $c\in \C P^1$,
\item $P_{13,14, 13}$  $\longrightarrow$ $(c_1,c_2,\frac{c_1}{c_2})$, $c_1,c_2\in \C P^1$,
\item $P_{13, 24, 34}$ $\longrightarrow$ $(\infty, c,\infty)$, $c\in \C P^1$,
\item $P_{14,23,24}$ $\longrightarrow$ $(0,0,c)$, $c\in \C P^1$,
\item $P_{14,23,34}$ $\longrightarrow$ $(c, 0,\infty)$, $c\in \C P^1$,
\item $P_{14,24,34}$ $\longrightarrow$ $(c_1,c_2,\frac{c_1}{c_2})$, $c_1,c_2\in \C P^1$,
\item $P_{23,24,34}$ $\longrightarrow$ $(c_1,c_2,\frac{c_1}{c_2})$, $c_1,c_2\in \C P^1$.
\end{itemize}
\end{lem}

The two-dimensional orbits in the chart $M_0$ are the orbits of those points from $\C ^{5}$ which have exactly four coordinates equal to zero. It implies:  

\begin{lem}
A stratum in the chart $M_{0}$ for $(\C ^{*})^4$ - action on $\C P^5$ over an interval can be parametrized by any $(c_1,c_2,\frac{c_1}{c_2})$, $c_1,c_2\in \C P^1$.
\end{lem}
We introduce slightly different notation of this parametrization using the notion of projective completion  of the complex plane, as it is standardly done to pass from affine to projective geometry~\cite{AU}. 
\begin{lem}
The parametrization  of the strata in the chart $M_1$ is equivalent to the following:
\begin{enumerate}
\item $(c_1,c_2,\frac{c_1}{c_2}) \longleftrightarrow (c_1 : c_2 : 1) \in \C P^2$, $c_1,c_2\in \C ^{*}$,
\item $(0,c_2,0)\longleftrightarrow (0 : c_2 : 1)$, $c_2\in \C ^{*}$,
\item $(c_1,0,\infty)\longleftrightarrow (c_1 :0 :1)$, $c_1\in \C ^{*}$,
\item $(\infty ,\infty ,c) \longleftrightarrow (c : 1 : 0)$ for $c\in \C ^{*}$
\item $(\infty, \infty, \infty) \longleftrightarrow (1: 0: 0)$.
\item $(0,0,c), c\in \C P^1 \longleftrightarrow (0 : 0 : 1)$,
\item $(c_1,c_2,\frac{c_1}{c_2})$, $c_1,c_2\in \C P^1 \longleftrightarrow (c_1 : c_2 : 1)\; \text{or}\; (c_1 : c_2 :0)$, 
\item $(\infty, c,\infty), c\in \C P^1 \longleftrightarrow (0: 1 :0)$,
\item $(c, 0,\infty)$, $c\in \C P^1 \longleftrightarrow (1: 0 :0)$,
\end{enumerate}
\end{lem}

In this way we obtain
\begin{prop}\label{chart}
For  the chart $M_0$ the orbit space $M_0/T^3$ is homeomorphic $\Delta _{4,2}\times (\C P^ 2-\C P^1)\cup P_{34}\times \C P^1$.
\end{prop}
\begin{proof}
It follows when summarize previous cases. 
The pyramids $P_{13}$, $P_{24}$ and the rectangle $P_{13,24}$  glue together to give $\Delta _{4,2}$ and they are parametrized by $(c_1 : 0 :1)$. Also the  pyramids $P_{14}$, $P_{23}$ and the rectangle $P_{14,23}$ glue together to  give $\Delta _{4,2}$ and they are parametrized by $(0 : c_2 :1)$. The tetrahedra $P_{13,14},P_{13,23}, P_{14,24}, P_{23,24}$  and the triangles $P_{13,14,23}$, $P_{13,14,24}$, $P_{13,23,24}$ and $P_{14,23,24}$ glue together to give $\Delta _{4,2}$ and they are parametrized by $(0:0:1)$.  The pyramid $P_{34}$ is parametrized by $(c:1:0)$, $c\in \C ^{*}$. Besides that the tetrahedra $P_{13,34}, P_{24,34}$ and the triangle $P_{13,24,34}$ glue together to give $P_{34}$ and they are parametrized by  $(1:0:0)$. The tetrahedra $P_{14,34}$, $P_{23,34}$ and the triangle $P_{14,23,34}$ glue together to give $P_{34}$ and they are parametrized by $(0:1:0)$. 
\end{proof}

We consider now the parametrization of the strata whose admissible polytopes do not contain the vertex $\delta_{12}$. Such strata do not belong to the chart $M_{0}$. The approach is the same as in the case of Grassmann manifold $G_{4,2}$. We demonstrate it for the admissible polytope $P_{34}$.

\begin{prop}\label{new}
The orbits of the stratum whose admissible polytope is $P_{34}$ can be parametrized, using the parametrization~\eqref{main} of the main stratum, by $(c : 1 : 0)$, where $c\in \C ^{*}$.
\end{prop}
\begin{proof}
This strata belongs to the chart $M_{1} = \{[(w_1,1,w_2,w_3,w_4,w_5)], w_i\in \C \}$. The $(\C ^{*})^4$-orbit of a point $[(a_1,1,a_2,a_3,a_4,a_5)$ is given by
\[
[(\tau _{1}a_1,1,\tau _{2}a_2,\tau _{3}a_3,\tau _{2}\tau _{3}a_4,\frac{\tau _{2}\tau _{3}}{\tau _{1}})],
\]
what implies that the main stratum  is, in this chart, given by the equation
\[
\frac{w_2w_3}{w_4} = d_{1},\;\; \frac{w_2w_3}{w_1w_5}=d_{2},
\]
where $d_1, d_2\in \C ^{*}$.
The parameters $(c_1,c_2)$ of the main stratum in the chart $M_{0}$ and $(d_{1},d_{2})$ in the chart  $M_{1}$ are related by
\[
c_1 = d_{2},\;\; c_{2} = \frac{d_2}{d_1}.
\]
The stratum whose admissible polytope is $P_{34}$ is, in the chart $M_{1}$, given by
\[
w_{1}=0,\;\; \frac{w_2w_3}{w_4} = w_{1},
\]
what implies that its orbits  are parametrized by $(d_1, 0, 0)$. It gives that using parametrization of the main stratum in the chart $M_{0}$ it can be parametrized by $(0, 0, c)$, $c\in \C ^{*}$ what is equivalent to $(c : 1: 0)$, where $c\in \C ^{*}$.
\end{proof}
In the same way we prove:
\begin{prop}\label{other}
Using parametrization of the main stratum:
\begin{itemize}
\item the rectangle $P_{34,12}$ can be parametrized by $(c : 1 :0)$, $c\in \C ^{*}$,
\item the tetrahedra $P_{13, 12}$, $P_{12, 24}$ and the triangle $P_{12,13,24}$ by $(1:0:0)$,
\item the tetrahedra $P_{12, 23}$, $P_{12,14}$ and the triangle $P_{12,14,23}$ by $(0:1:0)$.
\end{itemize}
\end{prop}

The Propositions~\ref{chart},~\ref{new} and~\ref{other} gives:
\begin{thm}\label{orbit-space-5}
The orbit space $\C P^5/T^4$ for the action of the compact torus $T^4$ on $\C P^5$ given by the composition of the second symmetric power  $T^4\hookrightarrow T^6$ and the canonical action of $T^6$ to $\C P^5$ is homeomorphic to 
\begin{equation}
(\Delta _{4,2}\times \C P^2)/\approx \; \text{where}\; (x, c)\approx (y,c^{'})\Leftrightarrow x=y \in \partial \Delta _{4,2}.
\end{equation}
\end{thm}

Arguing in the same  way as in the proof  of Corollary~\ref{join} we deduce:
\begin{cor}
The orbit space $\C P^5/T^4$ for the  action of the compact torus $T^4$ on $\C P^5$ given by the composition of the second symmetric power $T^4\hookrightarrow T^6$ and the canonical action of $T^6$ on $\C P^5$ is homeomorphic to $S^2\ast \C P^2$.
\end{cor}
\section{On $T^4$-pair $(\C P^5, G_{4,2})$}
The Plu\"cker coordinates  give the  standard embedding of $G_{4,2}$ into  $\C P^{5}$ for which we have: 
\begin{lem}
The embedding of $G_{4,2}$ into $\C P^5$ given by the Pl\"ucker coordinates is the complex hypersurface in $\C P^5$, which is, in the homogeneous coordinates $(z_1,\ldots ,z_6)$ for $\C P^5$,  given by the equation
\begin{equation}\label{hyper}     
z_0z_5 + z_2z_3 = z_1z_4.
\end{equation}
\end{lem}
\begin{proof}
Let $X\in G_{4,2}$ and $X$ is represented by a matrix
\[
A_{X}=\left(\begin{array}{cc}
a_1 & a_2\\
a_3 & a_4\\
a_5 & a_6\\
a_7 & a_8 
\end{array}\right) .   
\]
The embedding $G_{4,2}\hookrightarrow \C P^5$ given by the Pl\"ucker coordinates is $z_0=a_1a_4-a_2a_3, z_1=a_1a_6-a_2a_5, z_2=a_1a_9-a_2a_7, z_3=a_3a_6-a_4a_7, z_4=a_3a_8-a_4a_7, z_5=a_5a_8-a_6a_7$. The direct computation shows that it gives hypersurface defined by the equation~\eqref{hyper}.
\end{proof}

The embedding of $G_{4,2}$ into $\C P^5$ is equivariant under the standard action of $(\C ^{*})^4$ on $G_{4,2}$ and the action  of $(\C ^{*})^4$ on $\C P^5$  given by the embedding~\eqref{emb}. Therefore the hypersurface~\eqref{hyper} is invariant under the given action of $(\C ^{*})^{4}$. 
It implies that the sphere $S^5$ which is the orbit space $G_{4,2}/T^4$ is embedded into seven-dimensional orbit space $\C P^{5}/T^4$.

In this we obtain:
\begin{thm}\label{embedd}
The embedding $G_{4,2}\hookrightarrow \C P^5$ given by the Pl\"ucker embedding induces an  embedding $G_{4,2}/T^4\cong S^5\cong S^2\ast \C P^1 \hookrightarrow \C P^5/T^4\cong S^2\ast \C P^2$ given by the canonical embedding $\C P^1\hookrightarrow \C P^2$
\[
(c:1) \to (c : 1 : (1-c)),\;\; (1,0)\to (0, 0, 1).
\]
\end{thm}
\begin{proof}
We consider the points from the main stratum  for $G_{4,2}$ in the chart $M_{12}$. They are given by $a_1=a_4=1$, $a_2=a_3=0$ and $\frac{a_5a_8}{a_6a_7}=c$, $c\in \C-\{0,1\}$, $a_i\neq 0$, $5\leq i\leq 8$. They are  by~\eqref{hyper}  embedded in $\C P^5$ in the hypersurface $z_5+z_2z_3=z_1z_4$, such that $z_1z_4 = \frac{c}{1-c}z_5$ and $z_i\neq 0$.  The main stratum  for  $\C P^5$ which belongs to this hypersurface is, in the chart  $M_{0}$ for $\C P^5$,  given by $\frac{z_2z_3}{z_5} = c_1$, $\frac{z_1z_4}{z_5}=c_2$, $c_1,c_2\in \C -\{0\}$. It follows that on these hypersurfaces we have $c_2=\frac{c}{1-c}$ and $c_1=\frac{1}{1-c}$. Therefore the embedding of the orbit space of the main stratum for $G_{4,2}$ into the orbit space of the main stratum  for  $\C P^5$ is given by $(c:1)\to (\frac{1}{1-c}: \frac{c}{1-c}: \frac{1}{c}) = (c : 1 : (1-c))$. Since all  other orbits  for the both spaces are continiously  parametrized using the parametrization of the main strata it follows that the embedding   $G_{4,2}/T^3\cong S^2\ast \C P^1 \hookrightarrow \C P^5/T^3\cong S^2\ast \C P^2$ is given by the same formula. 
\end{proof}

\section{Results in the real case}
The action of the compact torus $T^4$ on the complex Grassmann manifold $G_{4,2}$ and the complex projective space $\C P^5$ induces naturally the action of the group $\Z _{2}^{4}$ on the real Grassmann manifold $G_{4,2}(\R )$ and the real projective space $\R P^5$. The importance for  the consideration of the real case is connected with the problematics and the result of~\cite{GFMC} and~\cite{TO}.   

Since all our constructions in the complex case given in the previous sections  are compatible with the involution given by the complex conjugation the results obtained in the complex case translates directly to the real case.  Therefore Theorem~\ref{orbit-space-4} implies:

\begin{thm}
The orbits space $X(\R)=G_{4,2}(\R)/\Z _{2}^{4}$ for the canonical action of  $\Z _{2}^{4}$ on $G_{4,2}(\R)$ is homeomorphic to the quotient space 
\begin{equation}
(\Delta _{4,2}\times \R P^1)/\approx \; \text{where}\; (x, r)\approx (y,r^{'})\Leftrightarrow x=y \in \partial \Delta _{4,2}.
\end{equation}
\end{thm}

\begin{cor}
The orbit space $G_{4,2}(\R)/\Z _{2}^{4}$ is homeomorphic to the join $S^2\ast S^1$ and, thus, homotopy equivalent to $S^4$.
\end{cor}

We can proceed in the same way as in the proof of Theorem~\ref{orbit-main} and Theorem~\ref{top-main} to obtain:

\begin{thm}
The orbit space $X(\R)=G_{4,2}(\R)/\Z _{2}^{4}$ is homeomorphic to the sphere $S^4$.
\end{thm}

As for the induced action of $\Z _{2}^{4}$ on $\R P^5$, from Theorem~\ref{orbit-space-5} it directly follows:

\begin{thm}
The orbit space $\R P^5/\Z _{2}^{4}$ for the action of $\Z _{2}^{4}$ on $\R P^5$ given by the composition of the second symmetric power  $\Z _{2}^4\hookrightarrow \Z _{2}^6$ and the canonical action of $\Z _{2}^6$ on $\R P^5$ is homeomorphic to 
\begin{equation}
(\Delta _{4,2}\times \R P^2)/\approx \; \text{where}\; (x, r)\approx (y,r^{'})\Leftrightarrow x=y \in \partial \Delta _{4,2}.
\end{equation}
\end{thm} 

\begin{cor}
The orbit space $\R P^5/\Z _{2}^4$ for the  action of  $\Z _{2}^4$ on $\R P^5$ given by the composition of the second symmetric power $\Z _{2}^4\hookrightarrow \Z _{2}^6$ and the canonical action of $\Z _{2}^6$ on $\R P^5$ is homeomorphic to $S^2\ast \R P^2$.
\end{cor}

It is also straightforward to see that Theorem~\ref{embedd} is true in the real case as well:

\begin{thm}
The embedding $G_{4,2}(\R)\hookrightarrow \R P^5$ given by the Pl\"ucker embedding induces an  embedding $G_{4,2}(\R)/\Z _{2}^{4}\cong S^4\cong S^2\ast \R P^1 \hookrightarrow \R P^5/\Z _{2}^4\cong S^2\ast \R P^2$ given by the canonical embedding $\R P^1\hookrightarrow \R P^2$
\[
(r:1) \to (r : 1 : (1-r)),\;\; (1,0)\to (0, 0, 1).
\]
\end{thm}

\subsection{The smooth and singular  points of the orbit space $G_{4,2}(\R )/\Z _{2}^{4}$}
The equivariant tubular neighbourhood theorem is valid for finite groups as well~\cite{BR} and in the analogous way as it is done in the complex case we describe the smooth and the singular points of the orbit space $G_{4,2}(\R )/\Z _{2}^{4}$.
Denote by  $\pi : G_{4,2}(\R)\to G_{4,2}(\R)/\Z _{2}^{4}$ the  natural projection.
\begin{thm}\label{srmain}
The point $q\in G_{4,2}(\R )/\Z _{2}^{4}$ is:
\begin{itemize}
\item a smooth point if $\pi ^{-1}(q) = \Z _{2}^{3}$;
\item a cone-like singularity point if $\pi ^{-1}(q)\cong \Z_2^k$, $0\leq k\leq 2$ and has  neighbourhood of the form 
\begin{enumerate}
\item $D^2\times \cone (\R P^1 )\;\;$ for $\;\; \pi ^{-1}(q)=\Z _{2}^{2}$;
\item  $D^1\times \cone (S^2/\Z _{2}^{2})\;\;$ for $\;\; \pi ^{-1}(q) =\Z _{2}$;
\item $\cone (S^3/\Z _{2}^{3})\;\;$ for $\;\; \pi ^{-1}(q)$ is a point.
\end{enumerate}
where the induced actions of $\Z_{2}^{2}$ and $\Z_{2}^{3}$ on $S^2$ and $S^3$ respectively are without fixed points.
\end{itemize}
\end{thm} 
\begin{proof}
We provide the proof for the sake of clearness. We use the notation of~\ref{gen-smth} and follow the proof of Theorem~\ref{smain}. Let $q\in G_{4,2}(\R )/\Z _{2}^{4}$ and let $H$ be a stabilizer of the points from $\pi ^{-1}(q)$. For a point $p\in \pi ^{-1}(q)$, the tangent bundle for the orbit $\Z _{2}^{4}\cdot p$ is zero-dimensional what implies that the slice $V$ along such orbit is four-dimensional. Consider an orthogonal decomposition  $V = V^{H} \oplus L$ for some $Z_{2}^{4}$-invariant metric. The theorem on equivariant neighbourhood for a compact group action implies that  there exists a neighbourhood of this orbit~\eqref{ndif}, whose orbit space has the same smooth and singular points as  $V^{H}\times \cone (S(L)/H)$, where $S(L)$ is the unit sphere in $L$. In this interpretation the following cases are possible. If $H$ is trivial then $V^{H}=V = D^{4}$ and such points $q\in G_{4,2}(\R )/\Z _{2}^{4}$ are the smooth ones. If $H$ is not trivial then $H= \Z _{2}^{k}$, where $1\leq k\leq 3$. For $H =\Z _{2}$ in analogy to the complex case we obtain that $\dim V^{H}=2$, what implies that $\dim L=2$. Therefore the points  $q \in G_{4,2}(\R )/\Z _{2}^{4}$  with this stabilizer have the neighbourhood which has the same smooth and singular points as  $D^{2}\times \cone (S^{1}/\Z _{2}) =
D^{2}\times \cone (\R P^{1})$. For $H=\Z _{2}^{2}$, we have that $\dim V =1$, what gives that the corresponding points  have the neighbourhood which has the same smooth and singular points as  $D^{1}\times \cone (S^{2}/\Z _{2}^{2})$. For $H=\Z _{2}^{3}$ we obtain the fixed points and in the orbit space they have the neighbourhood with the same smooth and singular points as  $\cone(S^{3}/\Z _{2}^{4})$. 
\end{proof} 

\begin{rem}
It is well known open problem about the existence of the smooth structures on the sphere $S^{4}$ different from the standard one. From Theorem~\ref{srmain} it follows that any  neighbourhood of a point from the orbits space $G_{4,2}(\R)/\Z_{2}^{4}\cong S^{4}$ whose stabilizer is non-trivial contains the cone-like singularity point. It implies that there is no smooth structure on the sphere $S^{4}$ for which the projection $\pi : G_{4,2}(\R) \to G_{4,2}(\R)/\Z _{2}^{4}$ would be a smooth map.
\end{rem}

\section{Some properties of $T^4$-action on $G_{4,2}$}
We derive from Section~\ref{structure} and  Section~\ref{orbit} some properties of the standard action of the compact torus $T^4$ on $G_{4,2}$ which we found to be crucial for the topological description of its orbit space. The following properties immediately follows:
\begin{lem}
The Grassmann manifold $M=G_{4,2}$ endowed with the canonical action of $T^4$ 
and the standard moment map satisfy the following properties:
\begin{enumerate} 
\item There exists a smooth atlas $(M_{ij},\varphi _{ij})$  for $M$ such that each chart $M_{ij}$  is  everywhere dense set in $M$ and  contains exactly one fixed point which maps to zero by the coordinate map.
\item For any  chart $(M_{ij}, \varphi _{ij})$ it is given  the characteristic homomorphism $\alpha _{ij} : T^3\to T^4$
such that the homeomorphism $\varphi _{ij}$ is $\alpha _{ij}$ - equivariant: 
\[
\varphi _{ij}(tm)= \alpha _{ij}(t) \varphi _{ij}(m),\;\; 
t\in T^3,\; m\in M_{ij}.
\]
\item For any characteristic homomorphism $\alpha _{ij} : \T ^3\to \T ^4$, the weight vectors are pairwise linearly independent.
\item The map $\mu$ gives the bijection between  the set of fixed points   and the set of vertices of the polytope $\Delta _{4,2}$. 
\end{enumerate}
\end{lem}

Denote by $C_{ij}$ the cone over the weight vectors in the chart $M_{ij}$ and by $C_{\Delta _{4,2},\delta_{ij}}$ the cone of the octahedron $\Delta _{4,2}$ at the vertex $\delta _{ij}$. The moment map in this local coordinates satisfies:

\begin{lem}
 For any chart $(M_{ij}, \varphi _{ij})$ there exist the maps  $\psi _{ij} : C_{ij}\to C_{\Delta _{4,2},\delta_{ij}}$ and $\mu _{ij} : \C ^{n}\to C_{ij}$ such that for the moment $\mu$ on $M_{ij}$ it holds
\[
\mu  =  \psi _{ij}\circ \mu _{ij}\circ \varphi _{ij}.
\]
\end{lem}
\begin{proof}
 Without loss of generality we provide the proof in the  chart $M_{12}$. The homomorphism $\alpha _{12} : T^{3}\to T^{4}$ is given by  
\[\alpha _{12}(t_1,t_2,t_3) = (t_1, t_2, t_3, \frac{t_2t_3}{t_1}).\] 
Therefore the corresponding  weight vectors are 
\[
\Lambda _{12}^{1} = (1,0,0), \Lambda _{12}^{2} = (0,1,0), \Lambda _{12}^{3} = (0,0,1), \Lambda _{12}^{4} = (-1,1,1).
\]
Let $V_j = \delta _{pq} - \delta _{12}$, where $\{p,q\}\neq \{1,2\}$,\; $1\leq j\leq 5$ . We assume here that $\{p,q\}\subset\{1,\ldots ,4\}$, $\{p,q\}\neq \{1,2\}$ are ordered lexicographically. Note that in such ordering it holds $V_5 = V_2+V_3 = V_1 + V_4$. 
Then  $\mu \circ \varphi _{12}^{-1} : \C ^{4} \to \Delta _{4,2}$ can be written as 
\[
(z_1,z_2,z_3,z_4) \to  \delta _{12}+ \frac{\sum _{j=1}^{4}|z_j|^{2}V_j +|z_1z_4-z_2z_3|^2V_5}{1+\sum _{j=1}^{4}|z_j|^2+|z_1z_4-z_2z_3|^2}
\]

Put ${\bf z} = (z_1,z_2,z_3,z_4)$ and consider $\xi _{12}: \C ^{4}\to \R ^{4}_{\geq 0}$ given by
\[
\xi _{12}({\bf z}) = \kappa _{0}({\bf z})\cdot (|z_1|, \sqrt{|z_2|^2+|z_1z_4-z_2z_3|^2}, \sqrt{|z_3|^2+|z_1z_4-z_2z_3|}, |z_4|),
\]
\[
\kappa _{0}({\bf z}) = \frac{1}{\sqrt{1+\sum_{j=1}^{4}|z_j|^2+|z_1z_4-z_2z_3|^2}}.
\]
There is the canonical map $\gamma _{12} : R_{\geq 0}^{4}\to C_{12}$ given  by
 \[\gamma _{12}(x_1,\ldots ,x_4) = \sum_{j=1}^4 x_j\Lambda_{12}^j. \] 
Define $\mu _{12} : \C ^4 \to C_{12}$ by
\[
\mu _{12} = \gamma _{12}\circ \xi _{12}.
\]
On the other hand let $C_{\Delta _{4,2},\delta _{12}}$ be the cone of $\Delta _{4,2}$ at the vertex $\delta _{12}$ and consider the  map $\psi _{12}^{j} : \{ r\Lambda _{12}^{j} | r\geq 0\} \to C_{\Delta _{4,2}, \delta _{12}}$ given by
\[
\psi _{12}^{j}(r\Lambda _{12}^j) = \delta_{12}+ r^2V_j, 1\leq j\leq 4.
\]  
There is the  canonical map $\eta _{12} : R_{\geq 0}^{4} \to C_{\Delta _{4,2},\delta _{12}}$  defined by
\[
\eta _{12}(x_1,\ldots, x_4) = \delta _{12}+\sum _{j=1}^{4}x_j^2V_j.
\]
We obtain  the map  $\psi _{12} : C_{12} \to C_{\Delta _{4,2}, \delta _{12}}$ defined by $\eta _{12} = \psi _{12}\circ \gamma_{12}$. The map $\psi _{12}$ gives  the  homeomorphism of the cones $C_{12}$ and $C_{\Delta _{4,2}, \delta _{12}}$. 
In this way we obtain 
\[
\mu  =   \psi _{12}\circ \mu _{12}\circ \varphi _{12}.
\] 
\end{proof}
From the definition of the moment map it follows:
\begin{prop}
Let $P$ be an admissible polytope of $G_{4,2}$  with the set of vertices $\delta _{ij}$ where $(i,j)\in S$ for a fixed subset $S\subset \{(1,2),\ldots ,(3,4)\}$. Then the stratum which correspond to $P$ can be obtained as
\[
W_{P} = (\cap _{(i,j)\in S}M_{ij})\cap _{(i,j)\notin S}(G_{4,2}-M_{ij}).
\] 
\end{prop}

Recall  that whenever $P\neq \Delta _{4,2}$ the corresponding stratum $W_{P}$ on $G_{4,2}$ consists of one $(\C ^{*})^4$-orbit. By Theorem~\ref{moment-map-polytope} the boundary of any $(\C ^{*})^4$-orbit on $G_{4,2}$ consists of the orbits of smaller dimensions and it holds:
\begin{lem}
\begin{itemize}
\item The characteristic function $\chi$ is constant on $W_{P}$ for any stratum $W_{P}$ on $G_{4,2}$,
\item If $W_{P^{'}}\subset \overline {W_{P}}$ then $\chi (W_{P}) \subset \chi (W_{P ^{'}})$.
\end{itemize}
\end{lem}

Since any stratum $W_{P}$ is invariant under the action of compact torus $T^4$ on $G_{4,2}$ and the  moment map $\mu : G_{4,2}\to \Delta _{4,2}$ is $T^4$ invariant as well, we have the mapping $\widehat{\mu} : W_{P}\to \stackrel{\circ}{P}$ for any stratum $W_{P}$. Moreover it holds: 
\begin{lem}
 The mapping $\mu : W_{P}/T^4\to \stackrel{\circ}{P}$ is a fibration for any stratum $W_{P}$.
\end{lem}
\begin{proof}
It follows from Theorem~\ref{orbits_chart} that for  a stratum $W_{P}$ the torus $(\C ^{*})^{\dim P}$ acts freely on $W_{P}$. Thus, the stratum $W_{P}$ which consists of one $(\C ^{*})^{4}$-orbit is a toric manifold of dimension $2\dim P$. Therefore, $W_{P}/T^{\dim P} = W_{P}/T^4$ is homeomorphic to $\stackrel{\circ}{P}$. The only non-orbit stratum is the main stratum $W_{\Delta _{4,2}}$ and by Proposition~\ref{orbits}, we see that $\mu : W_{\Delta _{4,2}}/T^4 \to \stackrel{\circ}{\Delta _{4,2}}$ is a fibration with a fiber homeomorphic to $\C -\{0,1\}$. 
\end{proof}
\begin{rem}
Note that for any $X\in G_{4,2}$ we also have   $\dim   \OO _{\C}(X) = 2\dim P$, where $\mu (\OO _{\C}(X)) = \stackrel{\circ}{P}$.
\end{rem}

The following properties of the action of $T^4$ on $G_{4,2}$ turns out to be important and they are contained into the foundations of the theory of $(2n,k)$-manifolds.

Let $h: \R ^{4}\to \R$ be a linear function, $h(x) = <x,\nu>$ for a fixed vector $\nu \in \R^4$.  It can be considered the function $\mu _{h}: G_{4,2}\to \R$ defined by $\mu _{h} = h\circ \mu$. The following holds. 

\begin{prop}
There is a linear function   $h : \R^{4}\to \R$ such that:
\begin{enumerate}
\item $h(\delta _{ij})\neq h(\delta_{kl})$ for any two vertices $\delta _{ij}$ and $\delta _{kl}$ of $\Delta _{4,2}$,
\item the composition $\mu _{h} : G_{4,2}\to \R$ is a Morse function whose critical points coincides with the fixed points for $T^3$-action on $G_{4,2}$.
\end{enumerate}
\end{prop}
\begin{proof}
Take $\nu = (1,2,4,8)$ and consider the corresponding function $h : \R ^{4}\to \R$ defined by $h(x)= x_1+2x_2+4x_3+8x_4$. It checks directly that $h(\delta _{ij})\neq h(\delta_{kl})$ for any two verices $\delta _{ij}$ and $\delta _{kl}$ of $\Delta _{4,2}$. Without loss of generality it is enough to show in the chart $M_{12}$ that $h\circ \mu$ is a Morse function whose critical point is $(0,0,0,0)$. If  $X\in M_{12}$ it can be represented by the matrix
\[
A_{X}=\left(\begin{array}{cc}
1 & 0\\
0 & 1\\
z_1 & z_3\\
z_2 & z_4 
\end{array}\right),   
\]
where $z_i=x_i+iy_i$, $1\leq i\leq 4$.
The computation gives that 
\[
(h \circ \mu)(X)  = \frac{A(X)}{B(X)+C(X)},
\]
where
\[
A(X)= 1+ 6(x_1^2+y_1^2)+10(x_2^2+y_2^2)+ 5(x_3^2+y_3^2) + 9(x_4^2+y_4^2) + 12C(X),
\]
\[
B(X)= 1+\sum_{i=1}^{4}(x_i^2+y_i^2),\]
\[
 C(X)=(x_1x_4-x_2x_3-y_1y_4+y_2y_3)^2+(x_1y_4-x_2y_3-x_3y_2+x_4y_1)^2.
\]
It follows that
\[
\frac{\partial (h\circ \mu )}{\partial x_i}(X) = a_i\cdot x_i\cdot A(X) - B(X)(2x_i + \frac{\partial C}{x_i}(X)),
\]
\[
\frac{\partial (h\circ \mu )}{\partial y_i}(X) = a_i\cdot y_i\cdot A(X) - B(X)(2y_i + \frac{\partial C}{y_i}(X)),
\]
where $a_1=12, a_2=20, a_3=10, a_4=18$. It implies that $\frac{\partial (h\circ \mu )}{\partial x_i}(X)=\frac{\partial (h\circ \mu )}{\partial y_i}(X)=0$ if and only if $x_i=y_i=0$, $1\leq i\leq 4$. Furthermore the direct computation shows that
\[
\frac{\partial ^{2}(h\circ \mu)}{\partial x_i^2}(X)=\frac{\partial ^{2}(h\circ \mu)}{\partial y_i^2}(X) = a_i,\; 1\leq i\leq 4,
\]
\[
\frac{\partial ^{2}(h\circ \mu)}{\partial x_ix_j}(X)=\frac{\partial ^{2}(h\circ \mu)}{\partial y_iy_j}(X)=0,\; i\neq j
\]
\[
\frac{\partial ^{2}(h\circ \mu)}{\partial x_iy_j}(X)=0, \; 1\leq i,j\leq 4
\]
Therefore the Hessian of the function $h\circ \mu$ is given by the diagonal matrix with non-zero diagonal entries, which implies that it is non-singular. 
\end{proof}

The following fact is standard:

\begin{lem}\label{embed}
Let  $F : G_{4,2}\to \R^{16}$ be defined by $F(A)= A\cdot (A^{*})^{T}$. The map $F$ gives an $T^{4}$-equivariant embedding of $G_{4,2}$ into $\R^{10}$ given by lower or upper diagonal matrix of $F(A)$.
\end{lem}
In this way we obtain the map $\nu : G_{4,2}\to \R^{4}$ given by $\nu (A) = \diag(F(A))$ what is   the diagonal of $F(A)$.
\begin{prop}\label{diag}
The map $\nu : G_{4,2}\to \R^{4}$ coincides with the standard moment map.
\end{prop}
\begin{proof}
Let $L\in G_{4,2}$ and  $x=(x_1,\ldots ,x_4)$, $y=(y_1,\ldots ,y_4)$ be its orthonormal frame  related to the standard Hermitian metric in $\C ^{4}$. Then $\nu (L)= (|x_1|^2+|y_1|^2,\ldots ,|x_4|^2+|y_4|^2)$. Let $A$ be a matrix whose column vectors are $x$ and $y$. The standard moment map is given by 
\[\mu (A) = \frac{1}{\sum\limits _{1\leq i<j\leq 4}|P^{ij}|^{2}}(|P^{12}|^2+|P^{13}|^2+|P^{14}|^2,\ldots ,|P^{14}|^2+|P^{24}|^2+|P^{34}|^2),\]
where $P^{ij}$ are the minors of the matrix $A$ determined by the $i$-th and $j$-th rows.

We have that 
\[ |P^{12}|^2+|P^{13}|^2+|P^{14}|^2 = |x_1y_2-x_2y_1|^2+|x_1y_3-x_3y_1|^2+|x_1y_4-x_4y_1|^2 =\]
\[ |x_1|^2(|y_2|^2+|y_3|^2+|y_4|^2) - x_1{\bar y_1}({\bar x_2}y_2+{\bar x_3}y_3+ {\bar x_4}y_4)- {\bar x_1}y_1(x_2\bar{y_2}+x_3\bar{y_3}+x_4\bar{y_4})+\]
\[|y_1|^2(|x_2|^2+|x_3|^2+ |x_4|^2= |x_1|^2+|y_1|^2\] 
and in analogous way  $\sum\limits_{i\neq k}|P^{ki}|^2 = |x_k|^2+|y_k|^2$ for any $1\leq k\leq 4$. On the other hand $\sum\limits_{i=1}^{4}(|x_i|^2+|y_i|^2) = \sum\limits_{k=1}^{4}\sum\limits_{i\neq 4}|P^{ki}|^2 = 2\cdot \sum\limits_{1\leq i<j\leq 4}|P^{ij}|^2$, what gives that $\sum\limits_{1\leq i<j\leq n}|P^{ij}|^2 = 1$. This implies that the first coordinate of the standard moment map coincides with the first coordinate for $\nu$. In the same way it follows that $\mu$ and $\nu$ coincide.
\end{proof}
Lemma~\ref{embed} and Proposition~\ref{diag} imply: 
\begin{cor}
There exists  equivariant embedding $f: G_{4,2}\to \R ^{10}$ such that:
\begin{itemize}
\item  $\R^{10} = \R ^{4}\times \R ^{6}$, where $T^4$ acts trivially on $\R ^{6}$.
\item for the projection $\pi : \R^{4}\times \R^{6}\to \R^{4}$ it holds $\mu = \pi \circ f$
\end{itemize}
\end{cor}
\begin{rem}
In analogous way appealing to Section~\ref{complex} it can be shown that the considered action of $T^4$ on $\C P^5$ satisfies the same listed properties as the action of $T^4$ on $G_{4,2}$.
\end{rem}

\bibliographystyle{amsplain}

Victor M.~Buchstaber\\
Steklov Mathematical Institute, Russian Academy of Sciences\\ 
Gubkina Street 8, 119991 Moscow, Russia\\
E-mail: buchstab@mi.ras.ru
\\ \\ \\ 

Svjetlana Terzi\'c \\
Faculty of Science, University of Montenegro\\
Dzordza Vasingtona bb, 81000 Podgorica, Montenegro\\
E-mail: sterzic@ac.me 

\end{document}